\documentclass[11pt]{article}

\usepackage{amssymb,amsmath,amsfonts}
\usepackage{graphicx,color,enumitem}
\usepackage{mathrsfs}  
\usepackage[normalem]{ulem}
\usepackage{amsthm} 
\usepackage{bm}
\RequirePackage[numbers]{natbib}

\usepackage{stmaryrd}

\usepackage{geometry}
 \usepackage{mathtools}
 \mathtoolsset{showonlyrefs}

\usepackage[colorinlistoftodos, textwidth=4cm, shadow]{todonotes}
\RequirePackage[colorlinks,citecolor=blue,urlcolor=blue]{hyperref}

\usepackage{bbm}
\usepackage{soul}

\renewcommand{\d}{{\rm d}}

\newcommand{\E}{{\mathbb E}}
\newcommand{\F}{{\mathbb F}}

\renewcommand{\P}{{\mathbb P}}

\newcommand{\R}{{\mathbb R}}
\renewcommand{\S}{{\mathbb S}}
\newcommand{\N}{{\mathbb N}}

\newcommand{\Acal}{{\mathcal A}}

\newcommand{\Scal}{{\mathcal S}}
\newcommand{\Tcal}{{\mathcal T}}

\newcommand{\Mid}{{\ \Big|\ }}

\newcommand{\Fc}{{\mathcal F}}

\newtheorem{theorem}{Theorem}

\newtheorem{definition}[theorem]{Definition}

\newtheorem{lemma}[theorem]{Lemma}

\theoremstyle{remark}
\newtheorem{remark}[theorem]{Remark}
\newtheorem{example}[theorem]{Example}

\numberwithin{equation}{section}
\numberwithin{theorem}{section}

\definecolor{darkgreen}{rgb}{0,0.7,0}

\newcommand{\iii}{{\vert\kern-0.25ex\vert\kern-0.25ex\vert}}

\newcommand{\T}{\top}

\newcommand{\ints}{\int_{\R_+}}
\newcommand{\intt}{\int_t^T}
\renewcommand{\c}{\alpha}

\def \ep{\hbox{ }\hfill$\Box$}

\begin{document}

\newcommand{\dd}{\mathrm{d}}
\newcommand{\proofspace}{\mbox{} \\*}
\newcommand{\entrecro}[1]{\left[ #1 \right]}
\newcommand{\entrebra}[1]{\left \{  #1 \right \}}
\newcommand{\entrepar}[1]{\left(  #1 \right) }
\renewcommand{\bar}{\overline}
\newcommand{\1}{\mathbf{1}} 
\newcommand{\oo}[1]{ {\fontsize{6}{6}\selectfont \textcircled{\raisebox{-0.79pt}{\hspace*{0.001pt} #1} }}  } 
\newcommand{\be}[1]{\begin{equation} #1 \end{equation}}
\newcommand{\bec}[1]{\begin{equation} \begin{cases} #1\end{cases} \end{equation}}
\newcommand{\bes}[1]{\begin{equation} \begin{split} #1\end{split} \end{equation}}
\newcommand{\B}[1]{\boldsymbol{#1}}
\newcommand{\cali}[1]{ \mathcal{#1}}
\newcommand{\produit}{\bigstar}

\newcommand{\Chi}{\scalebox{1.1}{$\chi$}}
\newcommand{\Pg}{P^g_t}
\newcommand{\Px}{P^x_t}
\newcommand{\PI}[1]{P^i_t(#1)}
\newcommand{\Pc}{P^c_t}

\newcommand{\inter}{\llbracket 1,n \rrbracket}
\newcommand{\actionspace}{\mathbf{A}}
\newcommand{\lambdaspace}[1][]{%
\ifthenelse{\equal{#1}{}}{C_\lambda([0,T], L^1(\mu_1 \otimes \mu_2))}{C_{#1}([0,T], L^1(\mu_1 \otimes \mu_2))}%
}
\newcommand{\var}{\mathrm{Var}}
\renewcommand{\P}{\mathbb{P}}

\newcommand{\scal}[2]{#1.#2^{\otimes2}}

\renewcommand{\c}{\alpha}

\title{Linear--Quadratic  control for a class of stochastic Volterra equations: solvability and  approximation}

\author{Eduardo ABI JABER \footnote{Universit\'e Paris 1 Panth\'eon-Sorbonne, Centre d'Economie de la Sorbonne, 106, Boulevard de l'H\^opital, 75013 Paris, \sf  eduardo.abi-jaber at univ-paris1.fr.}
\quad\quad Enzo MILLER \footnote{LPSM, Universit\'e de Paris, Building Sophie Germain, Avenue de France, 75013 Paris,  \sf  enzo.miller at polytechnique.org}
\quad\quad  Huy\^en PHAM \footnote{LPSM, Universit\'e de Paris,  Building Sophie Germain, Avenue de France, 75013 Paris, \sf pham at lpsm.paris  
The work of this author  is supported by FiME (Finance for Energy Market Research Centre) and the ``Finance et D\'eveloppement Durable - Approches Quantitatives'' EDF - CACIB Chair.
}
}

\maketitle

\begin{abstract}
We provide an exhaustive treatment of Linear--Quadratic control pro\-blems for a class of stochastic Volterra equations of convolution type, whose kernels are Laplace transforms of certain signed matrix measures which are not necessarily finite. These equations are in general neither Markovian nor semimartingales, and include the fractional Brownian motion with  Hurst index smaller than $1/2$ as a special case.  We establish    the correspondence of the initial problem with  a possibly infinite dimensional Markovian  one in a Banach space, which allows us to identify the Markovian 
controlled state variables. Using a refined martingale verification argument combined with a squares  completion technique, we prove that the value function is of linear quadratic form in these state variables with a linear optimal feedback control, depen\-ding on non-standard Banach space valued Riccati equations.  Furthermore,  we show that the value function of the stochastic Volterra optimization problem can be approximated by that of conventional finite dimensional Markovian Linear--Quadratic problems, which is of crucial importance for  numerical implementation.
\end{abstract}

\vspace{5mm}


\vspace{5mm}


\newpage

\tableofcontents

\section{Introduction}

Let us consider the basic problem of controlling the drift $\alpha$ of a real-valued Brownian motion $W$
\begin{align}  \label{dynLQclassi}
X_t^\alpha &=\;   \int_0^t \alpha_s ds + W_t,  \quad t \geq 0, 
\end{align}
in order to steer the system to zero with minimal effort by minimizing over a finite horizon the cost functional
\begin{align*}
J(\alpha) &= \; \E\Big[ \int_0^T \big( |X_t^\alpha|^2 + \alpha_t^2 \big) dt \Big]. 
\end{align*}
This problem fits into the class of linear-quadratic (LQ) regulator problem, and  can be explicitly solved by different methods including standard dynamic programming, maximum principle or spike variation methods relying on It\^o stochastic calculus. It is well-known, see e.g.   \cite[Chapter 6]{yong1999stochastic}, that the optimal control $\alpha^*$ is in linear feedback form  
with respect to the optimal state process $X^*$ $=$ $X^{\alpha^*}$: 
\begin{align*}
\alpha_t^* & = \; - \Gamma(t) X_t^*, \;\;\; 0 \leq t \leq T,
\end{align*}
where $\Gamma$ is a deterministic nonnegative function solution to a Riccati equation, actually expli\-citly given  by $\Gamma(t)$ $=$ $\tanh(T-t)$, 
and thus the associated optimal state process $X^*$ is a mean-reverting Markov process. 

Suppose now that the noise $W$ is replaced by a Gaussian process with memory, ty\-pically a fractional Brownian motion, or more generally by stochastic Volterra processes.  It is then natural to ask how the structure of the solution is modified, and how it can be derived, knowing that, in this case,  
usual methods for Markov processes and stochastic calculus for semimartingales  can no longer be applied.

Stochastic Volterra processes appear in different  applications for population dynamics, tumour growth, or energy finance, and provide suitable models  for dynamics with memory and delay,  see \cite{baretal11, GLS:90, sch06}.  These processes have known a growing interest in finance with the recent  empirical findings on  rough volatility in \cite{volatilityrough2014}.  Stochastic Volterra equations have been studied by numerous authors, see   \cite{AJLP17, MS:15, PP:90} and the references therein.

In this paper, we address  the optimal control of $d$-dimensional stochastic Volterra equations of the form: 
\begin{align} \label{eqVolterra}
X_t^\alpha & = \; g_0(t) + \int_0^t  K(t-s) \Big( b(s,X_s^\alpha,\alpha_s) ds + \sigma(s,X_s^\alpha,\alpha_s) dW_s \Big), 
\end{align}
where $g_0$ is a deterministic function and $K$ is a (convolution) matrix-valued kernel  of the form
\begin{align*}
K(t)=\int_{\R_+} e^{-\theta t} \mu(d\theta), \quad t >0,
\end{align*}
for some signed matrix measure $\mu$. Our framework covers the case of  the fractional kernel $K(t)$ $=$ 
$t^{H- 1/2}/\Gamma(H+1/2)$ with $H$ $\leq$ $1/2$, arising  from  the Mandelbrot-Van Ness representation of the fractional Brownian motion with Hurst index $H$. 
We shall mainly focus on the case where the coefficients $b$ and $\sigma$ are in linear form with respect to the state and control arguments, and the cost to be minimized is of 
linear-quadratic form.

Since the (controlled) stochastic Volterra process \eqref{eqVolterra} is neither Markovian nor a semimartingale,  it is natural to  consider Markovian lifts  for which suitable stochastic tools and control methods apply. Inspired by the Markovian representation of fractional Brownian motion introduced in 
\cite{carmona1998fractional}, and more recently generalized to several un--controlled stochastic Volterra equations in \cite{AJEE18b, cuchiero2018generalized, harms2019affine}, we establish the correspondence of the initial problem 
with a lifted  Markovian controlled system $(Y_t^\alpha)_{t\in [0,T]}$ taking its values in the possibly infinite-dimensional Banach space $L^1(\mu)$. 
Next, in the LQ case, i.e., when $b,\sigma$ are of linear form, and the cost function is  linear-quadratic, we prove by means of a refined martingale verification argument combined with a squares completion technique,  that the value function is of quadratic form while the optimal control is in linear feedback form  with respect to these lifted state variables.  The coefficients of the quadratic and linear form of the value function and optimal control are expressed in terms of a non-standard system of   integral operator Riccati equations  whose solvability (existence and uniqueness) is proved in \cite{abietal19b}. A related infinite-dimensional Riccati equation appeared in \cite{alfonsi2013capacitary} for the minimization problem of an energy functional defined in terms of a non-singular (i.e. $K(0)$ $<$ $\infty$) completely monotone kernel.
We stress that, although there exists several results for LQ control problems in infinite-dimension, and even for Volterra processes (see \cite{bonaccorsi2012optimal}), they cannot be applied in our Banach-space context as they only concern Hilbert spaces. As detailed above, the first  contribution of our paper lies in the rigorous derivation of the optimal solution for the stochastic Volterra control problem. A second important feature of our approach is to provide a natural approximation of such solution  by a suitable discretization of the measure $\mu$, leading to conventional finite-dimensional LQ control problems, which involve standard matrix Riccati equations that can be numerically implemented.


The paper is organized as follows.  In Section \ref{secpreli}, we formulate the control problem, justify the correspondence with the lifted Markovian system in the Banach space $L^1(\mu)$, and formally derive the Riccati equation.  Section \ref{secmain} presents the main results: 
\begin{enumerate}
	\item  \label{introi}
	the analytic expression and solvability of the value function and optimal control in terms of  
	a  Banach-space valued Riccati equation. We illustrate our general result  on the LQ regulator example mentioned in the beginning of the introduction with a fractional noise with Hurst parameter $H\leq 1/2$;
	\item \label{introii}	
	a general stability result  for the solution of the stochastic Volterra control problem with respect to the kernel and its application for the approximation of the solution.
\end{enumerate}
\vspace{-0.2cm}  In Section \ref{S:sde}, we prove a general existence result for SDEs with Lipschitz coefficients in Banach spaces, which is used in particular to get the existence of an optimal control for the LQ Volterra control problem.  In Section \ref{S:verification}, we  provide a refined martingale verification theorem for 
LQ control problem in our context,  which  mainly relies on  It\^o's formula for quadratic functions in Banach spaces.    
The proof of the  solvability result \ref{introi} is completed in Section~\ref{S:proof}, and that of the stability result \ref{introii}  is detailed in Section~\ref{secapprox}. 

\vspace{2mm}
\noindent \textbf{Related literature.} The optimal control of stochastic Volterra equations has been consi\-dered in \cite{yong2006backward} by maximum principle method  leading to a characterization of  the solution in terms of a backward stochastic Volterra equation for the adjoint process.  In \cite{agram2015malliavin}, the authors also use  the maximum principle together with Malliavin calculus to  obtain a co\-rresponding adjoint equation as a standard backward SDE.  Although the kernel considered in these aforementioned papers is not  restricted to be of convolution type, 
the required conditions do not allow singularity of $K$  at zero, hence excluding the case of a fractional kernel with parameter $H$ $<$ $1/2$.   More recently, an extended Bellman equation has been derived in \cite{han2019time} for the associated controlled  Volterra equation. 


The solution to the LQ control problem as in \eqref{dynLQclassi} with controlled drift and additive noise has been obtained in  \cite{klepetal03} when 
the noise is  a fractional Brownian motion with  Hurst parameter $H$ $>$ $1/2$, and in \cite{duncan2013linear} when the noise is a general Gaussian process with an optimal control expressed   
as the sum of the well-known linear feedback control for the associated deterministic linear-quadratic control problem and the prediction of the response of  the system to the future noise process. 
Recently, the paper \cite{wan18} investigated LQ problem of stochastic Volterra equations by providing characterizations of optimal control in terms of some forward-backward system, but leaving  aside their solvability, and under some coefficients assumptions that preclude  singular kernels such as  the fractional kernel with parameter $H$ $<$ $1/2$.  

{Finally, we mention that some financial problems such as optimal execution with transient Market impact \cite{alfonsi2013capacitary} and  hedging in the presence of Market impact with fractional Brownian motion \cite{bank2017hedging}  can be formulated as  controlled Volterra problems of linear-quadratic type. }

\vspace{2mm}

\noindent \textbf{Notations.} For a Banach space $\mathcal B$, $L^2([0,T],\mathcal B)$ denotes the space of measurable and square integrable functions from $[0,T]$ to $\mathcal B$. 

For any $d \times d_1$-matrix valued measure $\mu_1$ on $\R_+$,  we denote by  $|\mu_1|$  its total variation,  which is a scalar nonnegative measure,  refer to  \cite[Section 3.5]{GLS:90} for more details. The space $L^1(\mu_1)$ consists of $\mu_1$-a.e.~equivalence classes of $|\mu_1|$-integrable functions $\varphi:\R_+\to \R^{d_1}$ endowed  with the norm
$\|\varphi \|_{L^1(\mu_1)} =\int_{\R_+} |\mu_1|(d\theta) |\varphi(\theta)|$, where we identify the function $\varphi$ with its class of equivalence.  For any such $\varphi$ the integral
$$ \int_{\R_+} \mu_1(d\theta)\varphi(\theta) $$ 
is well defined by virtue of the inequality
$$ \left| \int_{\R_+} \mu_1(d\theta)\varphi(\theta)  \right |  \leq   \int_{\R_+} |\mu_1|(d\theta)|\varphi(\theta) | ,   $$ 
see \cite[Theorem 5.6]{GLS:90}.
If $\mu_2$ is a  $d \times d_2$-matrix valued measure, the space  $L^1(\mu_1 \otimes \mu_2)$ consists of $\mu_1\otimes \mu_2$-a.e.~equivalence classes of 
$|\mu_1|\otimes |\mu_2|$-integrable functions $\Phi:\R_+^2 \to \R^{d\times d}$  endowed  with the norm $\|\Phi\|_{L^1(\mu_1 \otimes \mu_2)}=\int_{\R^2_+} {|\mu_1|}(d\theta) 
|\Phi(\theta,\tau)| {|\mu_2|}(d\theta) < \infty.$ 
For any such $\Phi$, the integral 
$$ \int_{\R_+^2} \mu_1(d\theta)^\top \Phi(\theta,\tau) \mu_2(d\tau) $$ 
is again well defined by virtue of \cite[Theorem 5.6]{GLS:90}. Both $(L^1(\mu_1),\|\cdot\|_{L^1(\mu)})$ and $(L^1(\mu_1\otimes \mu_2), \|\cdot\|_{L^1(\mu_1\otimes \mu_2)})$ are Banach spaces, see \cite[Theorem 3.11]{rudin2006real}. 
We also denote by $L^\infty(\mu_1)$ the set of measurable functions $\psi$ $:$  $\R_+\to\R^{d_1}$, which are bounded $\mu_1$-a.e., and by  $L^\infty(\mu_1\otimes\mu_2)$ the set of  measurable functions $\Phi:\R_+^2 \to \R^{d\times d}$, which are bounded $\mu_1\otimes \mu_2$-a.e, {that we endow with their usual norms 
	$\|\psi \|_{L^\infty(\mu_1)}$ and $\|\Phi \|_{L^\infty(\mu_1\otimes\mu_2)}$.}

\section{Formulation of the problem and preliminaries} \label{secpreli}

Let $(\Omega,\Fc,\F=(\Fc_t)_{t \geq 0}, \P)$ be a filtered probability space  supporting a one dimensional Brownian motion $W$. Fix $T>0$ and $d,d',m \in \mathbb N$. We consider  a controlled $d$-dimensional stochastic Volterra equation
\begin{align}
X^{\alpha}_t &= g_0(t) + \int_0^t K(t-s)b(s, X^{\alpha}_s,\alpha_s)ds + \int_0^t K(t-s)\sigma(s, X^{\alpha}_s,\alpha_s)dW_s, 
\label{eq:sve}
\end{align}
where $\alpha$  is an element of the admissible set
\begin{align*}
\Acal = \left\{ \alpha: \Omega \times [0,T] \to \R^m  \mbox{ progressively measurable such that } \sup_{0\leq t\leq T} \E\left[ |\alpha_t|^4\right] < \infty \right\},
\end{align*}
$g_0:[0,T]\to \R^d$ is a measurable function,  $K: [0,T] \to \R^{d \times d'}$ is a measurable kernel, and $b,\sigma:[0,T]\times \R^d\times \R^m \to \R^{d'}$  are of affine form:
\begin{align*}
b(t, x,a)&=\beta(t) + Bx + Ca, \\
\sigma(t, x,a)&= \gamma(t) +  D x  + Fa,
\end{align*}
where  $B,D\in \R^{d'\times d}$, $C,F\in \R^{d' \times m}$, and $\beta, \gamma: [0,T] \to  \R^{d'}$ are mea\-surable functions.  
We are chiefly interested in the case where $K$ is the Laplace transform 
\begin{align}\label{eq:cmmu}
K(t)=\int_{\R_+} e^{-\theta t} \mu(d\theta), \quad t >0,
\end{align}
of a signed $d\times d'$--measure $\mu$ satisfying 
\begin{align}\label{eq:totalvar}
\int_{\R_+} \left( 1 \wedge \theta^{-1/2}\right) |\mu|(d\theta)  < \infty,
\end{align}
{and 
	\begin{align}\label{eq:mucont}
	\theta\mapsto \int_0^T e^{-\theta u } u^{\zeta-1}  du   \in L^1(|\mu|), \quad \mbox{for some } \zeta \in (0,1/4),
	\end{align}}
where $|\mu|$ denotes the total variation of $\mu$. While condition  \eqref{eq:totalvar}  does not exclude   $\mu_{ij}(\R_+)=\pm \infty$  for some $i\leq d$, $j\leq d'$, or equivalently  a singularity of  the kernel $K$  at $0$, it does ensure that $K \in L^2([0,T],\R^{d\times d'})$ {and that $|\mu|$ is $\sigma$-finite}, see Lemma~\ref{L:L2kernel}.  The former implies that the  stochastic  convolution
\begin{align}\label{eq:stoconv}
t \mapsto  \int_0^t K(t-s)\xi_s dW_s
\end{align}  
is well defined as an It\^o integral, for  every  $t \leq T$,  for any progressively measurable process $\xi$ such that 
$$ \sup_{t\leq T} \E\left[|\xi_t|^2\right]<\infty.$$ 
Indeed, 
\begin{align*}
\E \left[  \int_0^t |K(t-s)|^2 |\xi_s|^2 ds \right]\leq \|K\|_{L^2(0,T)}^2 \sup_{s\leq T} \E\left[|\xi_s|^2\right] < \infty,
\end{align*} 
for every $t\leq T$.  The convolution
$$   
t \mapsto \int_0^t K(t-s)\xi_s ds,
$$
is also well defined for  every $t\leq T $, by virtue of the Cauchy--Schwarz inequality.    {Condition \eqref{eq:mucont} would yield the existence of a continuous version.}

We can now make precise the concept of solution to the controlled equation \eqref{eq:sve}.  By a solution to \eqref{eq:sve}, we mean an $\F$-adapted process  $X^\alpha$ with continuous sample paths  such  that  \eqref{eq:sve}  holds for all $t\leq T$, $ \P$-almost surely.  Under \eqref{eq:cmmu}-\eqref{eq:totalvar}-\eqref{eq:mucont}, assuming that 
$\beta,\gamma$ are measurable and bounded,  Theorem~\ref{P:strong} shows that the controlled stochastic  Volterra equation \eqref{eq:sve} admits a unique continuous solution $X^{\alpha}$, for any  continuous input curve $g_0$,  and any admissible control  $\alpha \in \Acal$. Furthermore, it holds that
\begin{align}
\sup_{0\leq t\leq T}\E \left[|X^{\alpha}_t|^4\right] <\infty. \label{eq:estimate moment X}
\end{align}

\begin{remark}
	Notice that due to the possible singularity of the kernel $K$, and in contrast with standard stochastic differential equations, the solution $X^\alpha$ to the controlled stochastic  Volterra equation does not satisfy in general the usual square integrability condition of the form: $\E[\sup_{0\leq t\leq T}|X_t^\alpha|^2]$ $<$ $\infty$. For this reason, we impose the stronger  condition $\sup_{t\leq T}\E[|\alpha|_t^4]<\infty$ for the set of admissible controls $\Acal$, which will turn out to be crucial for the martingale verification result, see Section \ref{S:verification}.
	\ep
\end{remark}

\vspace{1mm}

We consider a  cost functional given by
\begin{align}
\label{eq:original_problem}
J(\alpha)  &= \E\Big[\int_0^T  f(X^{\alpha}_s,\alpha_s)d s\Big], 
\end{align}
where the running cost $f$ has the following  quadratic form
\begin{align} \label{fquadra}
f(x,\alpha) &=x^\top Q x + \alpha^\top N\alpha + 2 x^\top L,
\end{align}
for some $Q \in \S^d_+, N \in \S^m_{+}$ and $L\in \R^d$. Here $\S^d_+$ denotes the set of $d$-dimensional nonnegative symmetric matrices. 
Note that by virtue of \eqref{eq:estimate moment X}, $J(\alpha)$ is well defined for any $\alpha \in \Acal$. The aim is to solve
\begin{align}\label{eq:main problem}
V_0 = \inf_{\alpha \in \Acal} J(\alpha).
\end{align}	

Before going further, let us  mention several kernels of interest that satisfy \eqref{eq:cmmu}-\eqref{eq:totalvar}--\eqref{eq:mucont}.
\begin{example} \label{ex:kernels}
	\begin{enumerate}
		\item Smooth kernels: if $ |\mu_{ij}(\R_+)|<\infty$, for every $i=1,\ldots, d, j=1,\ldots,d'$, then 
		\eqref{eq:totalvar}--\eqref{eq:mucont} are  satisfied and $K$ is infinitely differentiable on $[0,T]$.  This is the case,  for instance, when $\mu(d\theta)=\sum_{i=1}^n c^n_i \delta_{\theta_i^n}(d\theta)$, for some $c^n_i \in \R^{d\times d'}$ and $\theta^n_i \in \R_+$, $i=1,\ldots,n$, which corresponds to 
		$$ K(t) = \sum_{i=1}^n c^n_i e^{-\theta^n_i t }.  $$
		\item The fractional kernel ($d=d'=1$)
		\begin{align}\label{eq:kernelfrac}
		K_H(t)=\frac{t^{H-1/2}}{\Gamma(H+1/2)},
		\end{align}
		for some $H \in (0,1/2)$,  which is the Laplace transform of 
		\begin{align}\label{eq:mufrac}
		\mu_H(d\theta)=\frac{\theta^{-H-1/2}}{\Gamma(H+1/2)\Gamma(1/2-H)} d\theta,
		\end{align}
		and more generally the Gamma kernel $K(t)=K_H(t){e}^{-\zeta t}$ for  $H \in (0,1/2)$ and $\zeta \in \R$
		for which 
		$$ \mu(d\theta)=\frac{(\theta-\zeta)^{-H-1/2}\mathbbm 1_{(\zeta,\infty)}(\theta)}{\Gamma(H+1/2)\Gamma(1/2-H)} d\theta.$$
		\item\label{ex:kernels:iii} If $K_1$ and $K_2$ satisfy~\eqref{eq:cmmu}, then so does $K_1+K_2$ and $K_1K_2$ with the respective measures 
		$\mu_1 + \mu_2 $  and $\mu_1 * \mu_2$. When $\mu_1$, $\mu_2$ satisfy \eqref{eq:totalvar}-\eqref{eq:mucont},  it is clear that $\mu_1+\mu_2$ also satisfies \eqref{eq:totalvar}-\eqref{eq:mucont}.  Condition  \eqref{eq:totalvar}  is  satisfied for the convolution   $\mu_1 * \mu_2$ provided  $\int_{[1,\infty)^2} (\theta +\tau)^{-1/2} \mu_1(d\theta)\mu_2(d\tau)$ $<$ $\infty$, which is the case for instance if either 
		$\mu_1(\R_+)$ or $\mu_2(\R_+)$ are finite. {Similarly for \eqref{eq:mucont}.}
		\item
		If $K$ is a completely monotone kernel, i.e.~$K$ is infinitely differentiable on $(0,\infty)$ such that $(-1)^n K^{(n)}(t)$ is nonnegative for each $t>0$, then, by Bernstein's theorem, there exists a nonnegative measure $\mu$ such that \eqref{eq:cmmu} holds, see   \cite[Theorem 5.2.5]{GLS:90}. 
	\end{enumerate}
	\ep
\end{example}

\subsection{Markovian representation}

The solution $X^{\alpha}$  of \eqref{eq:sve} is in general neither Markovian nor a semimartingale as illustrated by the   Riemann--Liouville fractional Brownian motion 
$$ t \mapsto \frac{1}{\Gamma(H+1/2)} \int_0^t (t-s)^{H-1/2}  dW_s,  \quad H\in (0,1/2],$$
which is Markovian and a martingale only for $H=1/2$.  {Markovian representations of fractional Brownian motion have been introduced in 
	\cite{carmona1998fractional}, and more recently generalized to several un--controlled stochastic Volterra equations for kernels of the form \eqref{eq:cmmu}, see  
	\cite[Section 4]{AJEE18b};  \cite[Section 5.1]{cuchiero2018generalized}; 
	\cite{harms2019affine}. Inspired by these approaches,} we establish in the following theorem, by means of   stochastic Fubini's theorem,  the correspondence of \eqref{eq:sve} with a possibly infinite dimensional Markovian controlled system of the form 
\begin{equation}  \label{eq:yalpha}
\left\{
\begin{array}{ccl}
dY^{\alpha}_t(\theta) &=  &\left(- \theta Y^{\alpha}_t(\theta) + \tilde{b}\left(t,\int_{\R_+}  \mu(d\tau)Y^{\alpha}_t(\tau), \alpha_t \right)\right)dt  \\ 
& & \quad\quad  + \;           \tilde{\sigma}\left(t,\int_{\R_+}  \mu(d\tau)Y^{\alpha}_t(\tau), \alpha_t \right)dW_t   \\
Y_0^{\alpha}(\theta) &=&  0,
\end{array}
\right.
\end{equation}
where the  coefficients  $\tilde b:[0,T]\times \R^d\times \R^m \to \R^{d'}$, $\tilde \sigma: [0,T]\times  \R^d \times \R^m \to \R^{d'}$ are defined by
\begin{align}
\tilde{b}(t, x,a)&=\tilde{\beta}(t) + Bx + Ca,  \label{eq:btilde} \\
{ \tilde{\sigma}(t, x,a)}&=\tilde{\gamma}(t) +  D x  + Fa, \label{eq:sigmatilde}
\end{align}
with
$$ \tilde \beta = \beta + Bg_0\; \mbox{ and } \; \tilde \gamma =  \gamma + Dg_0.$$

\begin{theorem}
	\label{prop:representation_of_X}  	Let $g_0,\beta,\gamma$ be bounded functions on $[0,T]$ and $K$ be given as in  \eqref{eq:cmmu} such that \eqref{eq:totalvar}-\eqref{eq:mucont} hold.  Fix $\alpha \in \Acal$. Assume that there exists a progressively measurable process {$X^{\alpha}$ that}  solves \eqref{eq:sve}, $\P$-a.s., for each $t\leq T$, and that \eqref{eq:estimate moment X} holds. Then,  for each $t\leq T$, $X^{\alpha}_t$ admits the representation 
	\begin{align}
	\label{eq:repX}
	X_t^{\alpha}&= g_0(t) + \int_{\R_+}  \mu(d\theta)Y^{\alpha}_t(\theta), 
	\end{align}
	where, for each $\theta\in \R_+$,
	\begin{align}\label{eq:defY}
	Y^{\c}_t(\theta) = \int_0^t e^{-\theta(t-s)} b(s,X_s^{\alpha},\alpha_s)ds + \int_0^t e^{-\theta(t-s)} \sigma(s,X_s^{\alpha},\alpha_s)dW_s.
	\end{align}	
	In particular, $Y^{\alpha}$ can be chosen to have continuous sample paths   in $L^1(\mu)$, {satisfying}
	\begin{align}
	\sup_{t\leq T } \E\left[ \| Y^{\alpha}_t\|^4_{L^1(\mu)}\right] &< \infty,  \label{eq:estimateY4} \\
	\sup_{t\leq T }\sup_{\theta \in \R_+} |Y_t^{\c}(\theta)|&<\infty,  \label{eq:estimateY4.2}
	\end{align}
	and for each $\theta \in \R_+$, $t\mapsto Y^{\alpha}_t(\theta)$   solves \eqref{eq:yalpha}.
	Conversely, assume that there exists  {a solution $Y^{\alpha}$ to  \eqref{eq:yalpha} that is continuous  in $L^1(\mu)$}, i.e.,  such that
	\begin{align} 
	Y^{\c}_t(\theta) &= \int_0^t e^{-\theta(t-s)} \tilde b\left(s,\int_{\R_+}  \mu(d\tau)Y^{\alpha}_s(\tau),\alpha_s\right)ds  \nonumber    \\ 
	&  \quad + \int_0^t e^{-\theta(t-s)} \tilde \sigma\left(s,\int_{\R_+}  \mu(d\tau)Y^{\alpha}_s(\tau),\alpha_s\right)dW_s, \quad  \P\otimes \mu-a.e.    \label{eq:Yintegmu}  
	\end{align}
	for each $t\leq T$,  and that \eqref{eq:estimateY4} holds. Then, the process $X^{\alpha}$ given  by \eqref{eq:repX} is  a continuous solution to    \eqref{eq:sve} such that  \eqref{eq:estimate moment X}   holds.
\end{theorem}

\begin{proof}
	Fix $t\leq T$ and set $Z^{\alpha}_t=\int_0^{t}b(s,X_s^{\alpha},\alpha_s)ds + \int_0^{t}\sigma(s,X_s^{\alpha},\alpha_s)dW_s  $. We first plug \eqref{eq:cmmu} in \eqref{eq:sve}  to get
	$$X_t^{\alpha} -g_0(t)= \int_0^t K(t-s)dZ^{\alpha}_s=  \int_0^t \left( \int_{\R_+} \mu(d\theta) e^{-\theta(t-s)}  \right) dZ^{\alpha}_s. $$
	An application of stochastic Fubini's, see  \cite[Theorem~2.2]{V:12},  yields 
	\begin{align*}
	\int_0^t \left( \int_{\R_+} \mu(d\theta) e^{-\theta(t-s)}  \right) dZ^{\alpha}_s = \int_{\R_+} \mu(d\theta)  \left(\int_0^t e^{-\theta(t-s)}   dZ^{\alpha}_s\right),
	\end{align*}
	where the interchange is possible since by Jensen's inequality on the normalized measure $(1\wedge \theta^{-1/2})\mu(d\theta)/ \int_{\R_+}(1\wedge \tau^{-1/2})\mu(d\tau)$ the term
	\begin{align*}
	\int_{\R_+}  \left(\int_0^t e^{-2\theta(t-s)} \E\left[d\langle Z^{\alpha}\rangle_s \right]  \right)^{1/2} |\mu|(d\theta) 
	\end{align*}
	is bounded from above, for some $c>0$ by,
	\begin{align*}
	c\sqrt{ \sup_{r \leq T} \left(  |\gamma(r)|^2  +   {\E\left[ |\alpha_r|^2\right]}+ \E\left[ |X^{\alpha}_r|^2 \right]\right)}  
	\left(\int_{\R_+}   \frac{ 1-e^{-2\theta t}}{2\theta}  \left(1\wedge \theta^{1/2}\right)|\mu|(d\theta)\right)^{1/2}
	\end{align*}
	which is finite due to
	the inequality 
	\begin{align}\label{eq:inequalityexp}
	\frac{\left( 1-e^{-2\theta t}\right)}{2\theta} \leq \frac 1 2 \left( 1 \vee 2t\right) \left(1 \wedge \theta^{-1} \right),
	\end{align} 
	condition \eqref{eq:totalvar},
	the boundedness of $\gamma$, the admissible set $\Acal$ and the estimate \eqref{eq:estimate moment X}. 
	The interchange is  justified similarly for the drift part. 
	{It} follows that 
	$$X_t^{\alpha}  = g_0(t)+  \int_{\R_+} \mu(d\theta)  \left(\int_0^t e^{-\theta(t-s)}   dZ^{\alpha}_s\right)= g_0(t) +  \int_{\R_+} \mu(d\theta)  Y^{\c}_t(\theta) $$
	where $Y^{\c}_t(\theta) $ is given by  \eqref{eq:defY} and corresponds to the variation of constants formula of \eqref{eq:yalpha}. 	The claimed continuity statement  together  with \eqref{eq:estimateY4}-\eqref{eq:estimateY4.2} are proved in Lemma~\ref{L:contuinuity Y}. The converse is proved along the exact same lines  by working them backward. 
\end{proof}

{The following remark justifies our choice for carrying the analysis in the space  $L^1(\mu)$.  
	\begin{remark}
		Inspecting the coefficients in \eqref{eq:yalpha}, one observes that they are well-defined if $Y^{\c}_t \in L^1(\mu)$ for all $t\leq T$. One might be tempted to look for solutions in the Hilbert space $L^2(\mu)$, but since $\mu$ is not always a finite measure (as in \eqref{eq:mufrac} for instance), $L^2(\mu)$ is not necessarily included in $L^1(\mu)$. 
\end{remark}}

\begin{remark}
	An alternative lift approach, in the spirit of \cite{AJEE18b,EER:07,han2019time, jacomg19, viezha18},  consists in  introducing the double-indexed (controlled) processes
	\begin{align*}
	G_t^\alpha(u) &= \E \left[ X_u^\alpha - \int_t^u K(u - s) b(s,X_s^\alpha,\alpha_s) ds \Mid  {\cal F}_t \right], \;\;\; 0 \leq t \leq u \leq T.
	\end{align*}
	The control problem can then be reformulated in terms of the infinite dimensional controlled Markov process ${\{G^\alpha_t(.), t \in [0,T]\}}$ with 
	It\^o dynamics
	\begin{align*}
	dG_t^\alpha(u) &= K(u-t) \left( b(t,X_t^\alpha,\alpha_t) dt + \sigma(t,X_t^\alpha,\alpha_t) dW_t \right), \;\;\; 0 \leq  t < u  \leq T.
	\end{align*}
	\ep
\end{remark}

\subsection{Formal derivation of the solution}\label{S:formal}
Thanks to  Theorem~\ref{prop:representation_of_X}, the   possibly non-Markovian initial problem can be formally recast as a degenerate infinite dimensional Markovian problem in $L^1(\mu)$ on the state variables $Y^{\alpha}$ given by \eqref{eq:yalpha}. To see this, we define the mean-reverting operator $A^{mr}$ acting on measurable functions $\varphi \in L^{1}(\mu)$ by
\begin{align} \label{meanoperator}
(A^{mr}\varphi)(\theta) &=-\theta \varphi(\theta), \quad \theta \in \R_+,
\end{align}
and consider  the dual pairing
\begin{eqnarray*} 
	\langle \varphi,  \psi \rangle_{\mu} &=& \int_{\R_+} \varphi(\theta)^\top \mu(d\theta)^\top \psi(\theta), 
	\quad  (\varphi,\psi)\in L^1(\mu)\times L^{\infty}(\mu^\top).   
\end{eqnarray*}
For any matrix--valued kernel $G$, we denote by $\boldsymbol{G}$ the integral operator induced by $G$, defined by: 
\begin{align*}
(\boldsymbol{G}\phi)(\theta)= \int_{\R_+} G(\theta,\theta')\mu(d\theta')\phi(\theta').
\end{align*}
Notice that when  $G$ $\in$ $L^\infty(\mu\otimes\mu)$, the operator  $\boldsymbol{G}$ is well-defined on $L^1(\mu)$, and we have $\boldsymbol{G}\phi$ $\in$ 
$L^\infty(\mu^\top)$ for  $\phi$ $\in$ $L^1(\mu)$. In this case, $\langle \phi, \boldsymbol{G}\psi \rangle_{\mu}$ is well defined for all $\varphi,\psi \in L^1(\mu)$.
When  $G$ $\in$ $L^1(\mu\otimes\mu)$, the operator $\boldsymbol{G}$ is well-defined on $L^\infty(\mu)$,  
and we have  $\boldsymbol{G}\phi$ $\in$ $L^1(\mu^\top)$, for $\phi$ $\in$ $L^\infty(\mu)$. 
In this case $\langle \boldsymbol{G}\phi ,\psi  \rangle_{\mu^\top}$ is well defined for all $\varphi,\psi \in L^\infty(\mu)$.

To fix ideas we set  $g_0=\beta=\gamma\equiv 0 $ and $L=0$. Noting that relation \eqref{eq:Yintegmu} is the mild form of the 
linear controlled dynamics in $L^1(\mu)$,
\begin{align*}
dY_t^{\alpha} &= \left( A^{mr}Y^{\alpha}_t +  \boldsymbol{B}Y^{\alpha}_t +  C\alpha_t \right) dt + \left( \boldsymbol{D}Y^{\alpha}_t +  F\alpha_t \right)dW_t, \quad Y_0^{\alpha} =0,
\end{align*}
we see that the optimization problem \eqref{eq:main problem}  can be reformulated as a Markovian problem in $L^1(\mu)$ with cost functional,  
\begin{align} \label{JL1mu}
J(\alpha)  &= \E\left[\int_0^T \left( \langle Y^{\alpha}_s , \boldsymbol{Q}Y^{\alpha}_s\rangle_{\mu}  + \alpha_s^\top N \alpha_s   \right)ds\right], \;\;\; 
\end{align}
where, by a slight abuse of notations,  $C$ and $F$ denote the respective constant operators from $\R^m$ into $L^\infty(\mu)$ induced by the matrices $C$ and $F$:
$$ (C a)(\theta)= C a, \quad (F a) (\theta)= F a, \quad \theta  \in \R_+,\quad a  \in \R^m. $$
Their adjoint operators $C^*,F^*$ from $L^1(\mu^\top)$ into $\R^m$  take the form 
$$   C^* g  \; = \; C^\top \int_{\R_+} \mu(d\theta)^\top g(\theta), \quad F^* g \; = \; F^\top \int_{\R_+} \mu(d\theta)^\top g(\theta), \quad   g \in L^1(\mu^\top). $$
Given the linear--quadratic structure of the problem, standard results in finite-dimensional stochastic control theory, see   \cite[Chapter 6]{yong1999stochastic},  as well as in Hilbert spaces, see \cite{flandoli1986direct, hu2018stochastic},   suggest that the optimal value process $V^{\alpha}$ associated to the functional \eqref{JL1mu}  should be of  linear--quadratic form 
$$ V_t^{\alpha^*} = \langle Y_t^{\alpha^*}, \boldsymbol{\Gamma}_{t} Y^{\alpha^*}_t \rangle_{\mu},  $$
with an   optimal feedback control $\alpha^*$ satisfying
\begin{align*}
\alpha^*_t = - \left( {N} + {F}^*\boldsymbol{\Gamma}_{t} {F} \right)^{-1} \left({C}^* \boldsymbol{\Gamma}_{t}  +  {F}^* \boldsymbol{\Gamma}_{t} \boldsymbol{D} \right) Y^{\alpha^*}_t, \quad 0 \leq t \leq T, 
\end{align*}
where $\boldsymbol{\Gamma}_{t}$ is  a symmetric operator from $L^1(\mu)$ into $L^\infty(\mu^\top)$,  and  solves the operator Riccati equation:
\begin{equation} \label{Riccatiformal}
\left\{
\begin{array}{ccl}
\boldsymbol{\Gamma}_{T} &=& \boldsymbol{0} \\
\dot    {\boldsymbol{\Gamma}}_{t}&=& - \boldsymbol{\Gamma}_{t}A^{mr} - \left(\boldsymbol{\Gamma}_{t}A^{mr}\right)^{*}-  \boldsymbol{Q} - \boldsymbol{D}^* \boldsymbol{\Gamma}_{t} \boldsymbol{D}   - \boldsymbol{B}^* \boldsymbol{\Gamma}_{t}  - \left(\boldsymbol{B^*}\boldsymbol{\Gamma}_{t}\right)^*  \\
& &  \quad + \;  \left({C}^* \boldsymbol{\Gamma}_{t} + {F}^* \boldsymbol{\Gamma}_{t} \boldsymbol{D} \right)^*\left( {N} + {F}^*\boldsymbol{\Gamma}_{t} {F} \right)^{-1} 
\left({C}^* \boldsymbol{\Gamma}_{t}  +{F}^* \boldsymbol{\Gamma}_{t} \boldsymbol{D} \right), \quad t \in [0,T].    
\end{array}
\right.
\end{equation}
In particular, when  $\boldsymbol{\Gamma}$ is an integral operator, this formally induces the following Riccati equation for the associated (symmetric)  kernel $\Gamma$ valued in 
$L^1(\mu\otimes\mu)$: 
\begin{equation} \label{formalkernelRiccati}
\left\{
\begin{array}{ccl}
{\Gamma}_{T}(\theta,\tau) &=&  {0}  \\
\dot    {{\Gamma}}_{t}(\theta,\tau) &=&  (\theta+\tau)\Gamma_{t}(\theta,\tau)  - Q   - D^\top \int_{\R_+^2} \mu(d\theta')^\top \Gamma_{t}(\theta',\tau')\mu(d\tau') D  \\
&  & \quad - \;  {B}^\top \int_{\R_+} \mu(d\theta')^\top \Gamma_{t}(\theta',\tau)  -   \int_{\R_+}  \Gamma_{t}(\theta,\tau') \mu(d\tau')B  +  \;  S_{t}(\theta)^\T \hat{N}_{t}^{-1} S_{t}(\tau),  
\end{array}
\right.
\end{equation}
where
\begin{align*}
S_{t}(\tau)  &= C^\T \ints \mu(d\theta)^\top \Gamma_{t}(\theta,\tau) + F^\top \int_{\R_+^2}  \mu(d\theta')^\top \Gamma_{t}(\theta',\tau')  \mu(d\tau') D   \\
\hat{N}_{t}  &= N + F^\T \int_{\R_+^2} \mu(d\theta)^\T \Gamma_{t}(\theta,\tau)   \mu(d\tau)  F,
\end{align*} 
and provides an  optimal control in the form
\begin{align*}
\alpha^*_{t} &= - \hat{N}_{t}^{-1} \int_{\R_+} S_{t}(\theta) \mu(d\theta)Y^{\alpha^*}_t(\theta), \quad 0 \leq t \leq T.
\end{align*}

Although the aforementioned infinite dimensional results provide formal expressions for the solution of the problem, they cannot be  directly applied, since they concern Hilbert spaces. Here the infinite dimensional controlled process $Y^{\c}$ takes its values in the  {non reflexive}  Banach space $\left(L^1(\mu),\|\cdot\|_{L^1(\mu)}\right)$. 
The rigorous derivation of the solution is the first main objective of the present paper. Our second goal is to show how to obtain an analytic finite-dimensional approximation of the original control problem after a suitable discretization of the operator Riccati equation.

\section{Main results} \label{secmain} 

We collect in this section our main results. 

\subsection{Solvability: optimal control and value function}

Let $\alpha \in \Acal$. Given the linear-quadratic structure of the problem and the formal analysis of Section \ref{S:formal}, it is natural to consider a  candidate optimal value process 
$(V_t^{\c})_{t \leq T}$ of linear-quadratic form in the state variable $Y^{\alpha}$ given by \eqref{eq:Yintegmu}, that is 
\begin{align}
V_t^{\alpha} &= \int_{\R_+^2} Y^{\alpha}_t(\theta)^\top \mu(d\theta)^\top \Gamma_{t}(\theta,\tau) \mu(d\tau)Y_t^{\alpha}(\tau)   +   
2\int_{\R_+} \Lambda_{t}(\theta)^\top    \mu(d\theta) Y_t^{\alpha}(\theta) + \chi_{t}, \label{eq:valuefunY} 
\end{align}
where the functions $t\mapsto \Gamma_{t}, \Lambda_{t}, \chi_{t}$ are  solutions, in a suitable sense, of the following system of Riccati equations:
\begin{equation} \label{eq:Riccatis_monotone}
\left\{
\begin{array}{rclc}
\dot{\Gamma}_{t}(\theta,\tau)  &=&  (\theta+\tau)\Gamma_{t}(\theta,\tau) -  \mathcal {R}_1 (\Gamma_{t})(\theta,\tau),  & \quad  \Gamma_{T}(\theta, \tau) = 0  \\
\dot \Lambda_{t}(\theta) &=& \theta \Lambda_{t}(\theta)  - \mathcal {R}_2 (t, \Gamma_{t}, \Lambda_{t})(\theta), &  \quad  \Lambda_{T}(\theta) =0\\
\dot \chi_{t} &=&     \; -  \mathcal {R}_3 (t, \Gamma_{t}, \Lambda_{t}), & \quad  \chi_{T} =0,
\end{array}
\right.
\end{equation}
where we defined
\begin{align}
\mathcal {R}_1(\Gamma) (\theta,\tau)&= \; Q +D^\top \int_{\R_+^2} \mu(d\theta')^\top \Gamma(\theta',\tau') \mu(d\tau') D
+ B^\top \int_{\R_+} \mu(d\theta')^\top \Gamma(\theta', \tau) \\ 
&\quad\quad  + \;  \ints \Gamma(\theta, \tau') \mu(d \tau') B   - S(\Gamma)(\theta)^\top \hat{N}^{-1}(\Gamma) S(\Gamma)(\tau)  \label{eq:R1}\\
\mathcal {R}_2 (t, \Gamma,  \Lambda )(\theta)  &= \;  L + Q g_0(t) + B^\top\ints \mu(d \theta')^\T \Lambda(\theta') + \ints \Gamma(\theta,\tau')\mu(d \tau') \tilde{\beta}(t)  \\
& \;   + \;  D^\T \int_{\R_+^2} \mu(d\theta')^\T \Gamma(\theta',\tau') \mu(d\tau') \tilde{\gamma}(t)   - S(\Gamma)(\theta)^\T \hat{N}(\Gamma)^{-1}  h(t,\Gamma, \Lambda) \label{eq:R2} \\
\mathcal {R}_3 (t, \Gamma,  \Lambda ) &= \;   g_0(t)^\T Q g_0(t)  + { 2L^\T g_0(t)} + \tilde{\gamma}(t)^\T \int_{\R_+^2} \mu(d\theta')^\T \Gamma(\theta',\tau') \mu(d\tau')\tilde{\gamma}(t)  \\
& \quad\quad  +   \; {2} \tilde{\beta}(t)^\T \ints \mu(d\theta')^\T \Lambda(\theta') - h(t,\Gamma, \Lambda)  \hat{N}(\Gamma)^{-1} h(t,\Gamma, \Lambda), \label{eq:R3}
\end{align}
with 
\bes{
	S(\Gamma)(\tau)  & = \;   C^\T \ints \mu(d\theta)^\top \Gamma(\theta,\tau)  + F^\top \int_{\R_+^2}  \mu(d\theta' )^\top
	\Gamma(\theta',\tau')  \mu(d\tau') D \\
	\hat{N}(\Gamma) & = \;  N + F^\T \int_{\R_+^2} \mu(d\theta)^\T \Gamma(\theta,\tau)   \mu(d\tau)  F \label{eq:hatN}\\
	h(t, \Gamma, \Lambda) & = \;  C^\T  \ints \mu(d\theta)^\T \Lambda(\theta)  +  F^\T \int_{\R_+^2} \mu(d\theta)^\T  \Gamma(\theta,\tau)    \mu(d\tau)\tilde{\gamma}(t).
}

The two following definitions specify the concept of  solution to  the system \eqref{eq:Riccatis_monotone}. 

\begin{definition}\label{D:nonnegative}
	Let $\Gamma:\R_+^2\to \R^{d\times d}$ such  that  $\Gamma \in L^\infty(\mu \otimes \mu)$.  We say that $\Gamma$ is symmetric if
	\begin{align*}
	\Gamma(\theta,\tau) &= \; \Gamma(\tau,\theta)^\top, \quad \mu\otimes \mu-a.e.
	\end{align*}  
	and  nonnegative if 
	\begin{align*}
	\int_{\R_+^2} \varphi(\theta)^\top \mu(d\theta)^\top	\Gamma(\theta,\tau) \mu(d\tau) \varphi(\tau) &\geq \; 0, \quad \mbox{ for all }    \varphi \in  L^{1}(\mu). 
	\end{align*} 
	We denote by $\S_+^d(\mu \otimes \mu)$ the set of all  symmetric and nonnegative  $\Gamma \in L^\infty (\mu \otimes \mu)$. 
\end{definition}


\begin{remark}
	The integral operator $\boldsymbol{\Gamma}$ associated to a symmetric kernel {$\Gamma $} $\in L^\infty(\mu \otimes \mu)$ 
	is  symmetric, in the sense that 
	\begin{align*}
	\langle \varphi,  \boldsymbol{\Gamma} \psi \rangle_{\mu}, &= \; \langle    \psi, \boldsymbol{\Gamma} \varphi\rangle_{\mu},  
	\quad  \varphi, \psi \in L^1(\mu).   
	\end{align*}
	Moreover, the nonnegativity of $\Gamma$ translates into 
	\begin{align*}
	\langle  \varphi,  \boldsymbol{\Gamma} \varphi \rangle_{\mu} & \geq \; 0, \quad  \varphi \in L^{1}(\mu). 
	\end{align*}
	\ep
\end{remark}


\begin{definition}
	By a solution to the system \eqref{eq:Riccatis_monotone}, we mean a triplet 
	$(\Gamma,\Lambda,\chi) \in C([0,T],L^1(\mu\otimes \mu))\times C([0,T],L^1( \mu^\top)) \times C([0,T],\R) $ such that  
	\begin{align}
	\Gamma_{t}(\theta,\tau)  &= \;   \int_t^T e^{-(\theta+\tau)(s-t)}     \mathcal {R}_1 (\Gamma_{s})(\theta,\tau) d  s,  &&  0\leq t\leq T, \quad  \mu\otimes \mu-a.e.
	\label{eq:Riccati_monotone_kernel_mild}\\
	\Lambda_{t}(\theta) &= \;   \int_t^T e^{-\theta(s-t)} \mathcal {R}_2 (s, \Gamma_{s},  \Lambda_{s} )(\theta) d  s, && 0\leq t\leq T, \quad    \mu-a.e.
	\label{eq:BSDE_monotone_kernel_mild}\\
	\chi_{t} &= \; \int_t^T \mathcal {R}_3 (s, \Gamma_{s},\Lambda_{s}) ds,  &&   0 \leq  t\leq T, 
	\label{eq:ODE_monotone_kernel_mild} 
	\end{align}
	where $\mathcal R_1$, $\mathcal R_2$ and $\mathcal R_3$ are defined  respectively by \eqref{eq:R1}, \eqref{eq:R2} and \eqref{eq:R3}.  
	In particular  $\hat N(\Gamma_{t})$ given by \eqref{eq:hatN} is  invertible for all $t\leq T$.
\end{definition}

\vspace{1mm}

The existence and uniqueness of a solution to the Riccati system follows from \cite{abietal19b}, and is stated in the next theorem. {The results in \cite{abietal19b} provide a solution to the equation for $\Gamma$, and we show how to derive the solution  for $\Lambda$ from \cite{abietal19b} in   Section~\ref{S:proof}.} 

\begin{theorem}\label{T:Riccatiexistence}
	Let $g_0,\beta,\gamma$ be bounded functions on $[0,T]$. Assume that $\mu$ satisfies \eqref{eq:totalvar} and that 
	\begin{align}
	\label{assumption:QN}
	Q \in \S^d_+ , \quad  N-\lambda I_m \in  \S^m_+,
	\end{align}
	for some $\lambda>0$. Then, there exists a {unique} triplet $(\Gamma,\Lambda,\chi) \in C([0,T],L^1(\mu \otimes \mu))\times C([0,T],L^1( \mu^\top)) \times C([0,T],\R)$ 
	to the system of Riccati equation \eqref{eq:Riccatis_monotone}
	such that \eqref{eq:Riccati_monotone_kernel_mild}, \eqref{eq:BSDE_monotone_kernel_mild}, \eqref{eq:ODE_monotone_kernel_mild}   hold  
	and $\Gamma_t \in \S^d_+(\mu \otimes \mu)$, for all $t \leq T$. 
	Furthermore,   there exists some positive constant $M>0$ such that  
	\begin{align}
	\label{eq:estimateRiccati}
	\int_{\R_+} |\mu|(d\tau) |\Gamma_{t}(\theta,\tau)| & \leq \;  M, \quad {\mbox{for $\mu$-almost every $\theta$}}, \quad 0 \leq t\leq T.
	\end{align}
\end{theorem}

\begin{remark}\label{R:Linfty} 
	Since $\Gamma_t \in \S^d_+(\mu \otimes \mu)$, we have $\Gamma_t \in L^1(\mu\otimes\mu)\cap L^{\infty}(\mu\otimes \mu)$, for all $t\leq T$. Similarly, $\Lambda_t \in L^1(\mu^\top) \cap L^{\infty}(\mu^\top)$. To see this, it suffices to observe that since $\Lambda \in C([0,T],L^1( \mu^\top))$, it is bounded in $ L^1(\mu^\top)$. Combined with the boundedness of $\Gamma$  in $L^1(\mu\otimes\mu)$, the estimate \eqref{eq:estimateRiccati} and the boundedness of the coefficients, we obtain
	\begin{align} 
	|\mathcal R_1(\Gamma_t)(\theta,\tau)|  +  |\mathcal R_2(t,\Gamma_t,\Lambda_t)(\theta)|   \leq  c, \quad \mu\otimes \mu-a.e., \quad t\leq T, 
	\end{align}
	for some constant $c$. Finally, from \eqref{eq:BSDE_monotone_kernel_mild}, we get that $\Lambda_t \in L^{\infty}(\mu^\top)$, for all $t\leq T$. 
	\qed
\end{remark}

\begin{remark}
	\label{R:continuityV}
	The process $V^{\alpha}$ given by \eqref{eq:valuefunY} is well-defined and continuous, due to the continuity of $(\Gamma,\Lambda,\chi)$, that of $Y^{\alpha}$ from Theorem~\ref{prop:representation_of_X} together with the bounds \eqref{eq:estimateY4.2}  and Remark \ref{R:Linfty}. 
	\ep
\end{remark}

\vspace{1mm}

Our first main result  addresses the  solvability of the problem \eqref{eq:main problem}. Theorem~\ref{T:mainoptimalsolution} esta\-blishes the existence of an optimal feedback control of linear form and provides an explicit  expression for the value function in terms of the solution to the Riccati equation. 
The proof is collected in Section~\ref{S:proof} and builds upon  the results developed in Sections \ref{S:sde} and  \ref{S:verification}.

\begin{theorem}\label{T:mainoptimalsolution}
	Let $\beta,\gamma$ be bounded functions on $[0,T]$ and $g_0$ be continuous. Fix  $K,\mu$ as in  \eqref{eq:cmmu}-\eqref{eq:totalvar}--\eqref{eq:mucont}. Under \eqref{assumption:QN}, let  
	$(\Gamma, \Lambda, \chi)$  be the solution to the system of Riccati equation \eqref{eq:Riccatis_monotone} produced by  Theorem~\ref{T:Riccatiexistence}. 
	Then, there exists an admissible control $\alpha^* \in \Acal$  with  corresponding controlled trajectory $Y^{\alpha^*}$ as in \eqref{eq:Yintegmu} such that 
	\begin{align}	\label{eq:optimalcontrol_monotone}
	\c^*_t &= \;  {-\hat{N}(\Gamma_{t})^{-1} \Big( h(t,\Gamma_{t}, \Lambda_{t}) + \int_{\R_+} S(\Gamma_{t})(\theta) \mu(d\theta) Y^{\alpha^*}_t(\theta)  \Big)}
	\end{align}
	for all $t \leq T$. Furthermore, $\alpha^*$ is an admissible optimal control, in the sense that 
	$$ \inf_{\alpha \in \mathcal A}  J(\alpha)=J(\alpha^*),$$
	$Y^{\alpha^*}$ is  the optimally controlled trajectory of the state variable and $V_t^{\alpha^*}$ given by \eqref{eq:valuefunY} is the optimal  value process of the problem, that is  
	\begin{align}	\label{eq:optimalvalue_monotone}
	V_t^{\alpha^*} &= \;  \inf_{\alpha \in \Acal_t(\alpha^*)} \E\left[\intt  f(X^{\alpha}_s,\alpha_s) ds  \Mid \cali{F}_t\right], \quad 0 \leq t \leq T, 
	\end{align}
	where $\Acal_t(\alpha)$ $=$ $\{ \alpha' \in \Acal: \alpha_s' = \alpha_s, \; s \leq t\}$. 
\end{theorem}

\begin{remark}
	From \eqref{eq:optimalvalue_monotone}, it follows that at initial time $t$ $=$ $0$, the optimal value $V_0$ is equal to $V_0$ $=$ $V_0^{\alpha^*}$ $=$ $\chi_0$, hence
	\begin{align*}
	V_0 & = \; \int_0^T  \mathcal {R}_3 (t, \Gamma_{t},\Lambda_{t}) dt.
	\end{align*}
	In particular,  for  a constant initial condition  $g_0(t)$ $\equiv$ $X_0$ for some $X_0$ $\in$  $\R^d$,  we have 
	\begin{align*}
	V_0 &= \;  X_0^\top \Psi(T) X_0 + \Phi(T) X_0  +  \xi(T),   
	\end{align*}
	for suitable functions  $\Psi,\Phi,\xi$, which corresponds to the usual linear--quadratic form  in $X_0$. However, because of the possible  non-Markovianity of the problem, for $t>0$, the optimal value $V_t^{\alpha^*}$  is not necessarily  linear--quadratic in $X^{\alpha^*}_t$ as in the standard case.  
	\ep
\end{remark}

\vspace{1mm}

The following example treats the LQ regulator problem \eqref{dynLQclassi} with a general Volterra noise. 
\begin{example}  Let us consider a controlled equation with Volterra noise 
	\begin{align*}
	X_t^\alpha  &=  \int_0^t \c_s d s  +  \int_0^t \tilde K(t-s) d W_s ,
	\end{align*}
	\bes{
		J(\c) = \E \entrecro{\int_0^T \big( QX_s^2 + N\c_s^2 \big) d s},
	}
	where $\tilde K(t) = \ints e^{-\theta t}\tilde\mu(d\theta)$. Notice that $X$ can be recast as 
	\bes{
		X_t^{\c} = \int_0^t K(t-s) \left (C \c_s ds + \gamma dW_s \right),
	}
	where $K$ is the row vector $(1, \tilde K)$, $C = (1,0)^\T$ and $\gamma = (0,1)^\T$. The kernel $K$ is the Laplace transform of the $1\times 2$-matrix measure $\mu$ $=$ $\left( \delta_0(d \theta), \tilde\mu(\d \theta)\right)$. An application of Theorem \ref{T:mainoptimalsolution} gives an  optimal control of feedback form in $Y$:
	\begin{align} \label{alpha*ex}
	\alpha^*_t &=  -\frac{1}{N} \Big[ \Gamma_t(0,0) Y_{t}^1(0) + \int_{\R_+} \Gamma_t(\theta,0) Y_t^2(\theta)\tilde\mu(d\theta) \Big], 
	\end{align}
	where $\Gamma$ is solution to the real-valued infinite dimensional Riccati equation  
	\bes{
		\Gamma_t(\theta, \tau) &= \int_t^T e^{-(\theta + \tau)(s-t)} \left(Q - \Gamma_s(\theta, 0) N^{-1} \Gamma_s(0, \tau)\right)ds, \quad 	 
		\tilde\mu\otimes\tilde\mu-a.e., \quad t \in [0,T],
	}
	and $Y_t(\theta)$ $=$ $(Y_t^1(\theta),Y_t^2(\theta))^\top$ $=$ $\int_0^t e^{-\theta(s-t)} \left( C \c_s^* ds + \gamma dW_s \right)$. In particular, 
	\begin{align*}
	Y_t^2(\theta) &= \int_0^t e^{-\theta(t-s)} dW_s \\
	Y_t^1(0) &= \int_0^t \alpha_s^* ds \; = \; X_t^* -  \int_0^t \tilde K(t-s) d W_s \; = \; X_t^* - \int_{\R_+} Y_t^2(\theta) \tilde\mu(d\theta),
	\end{align*}	
	where $X^*$ $=$ $X^{\alpha^*}$, and  the last equality holds by stochastic Fubini's theorem. Plugging  the expressions of $Y^1$ and $Y^2$  into \eqref{alpha*ex} yields 	
	\bes{
		\alpha^*_t & = -\frac{1}{N} \entrepar{ \Gamma_t(0,0)X_t^* + \ints \left(\Gamma_t(\theta,0) -  \Gamma_t(0,0)  \right) Y_{t}^2(\theta) \tilde\mu(d \theta) },
	}
	which is the sum of a feedback form in $X^*$, and a second term capturing the non Markovianity of $X$, as for example  in the case of a fractional noise 
	with Hurst parameter $H$ $\leq$ $1/2$.  
	Note that $\Gamma(0,0)$ satisfies the standard Riccati equation in LQ control problem. 	
	One can also note that when $\tilde K \equiv 1$ then $\tilde\mu(d\theta) = \delta_0(d\theta)$, which implies that the optimal control takes the standard feedback form 
	$\alpha^*_t = -\frac{1}{N}  \Gamma_t(0,0)X_t^*$. 
	\ep
\end{example}

\begin{remark}
	Conventional linear--quadratic models,  see for instance \cite[Chapter 7]{yong1999stochastic},  are  naturally nested in our framework. Indeed, they are recovered by setting   $d=d'$ and $\mu=\delta_{0} I_{d}$, which corresponds to  $K(t)\equiv I_d$. In this case, the Riccati equations for $\Gamma(0,0),\Lambda(0), \chi$ reduce to the conventional matrix Riccati equations and 
	$Y^{\alpha}=X^{\alpha}$ so that we recover the usual expression for the optimal control \eqref{eq:optimalcontrol_monotone} and the value function
	\bes{
		\c^*_t &= -\hat{N}(\Gamma_{t}(0))^{-1} \left( h(t,\Gamma_{t}(0,0), \Lambda_{t}(0)) +  S(\Gamma_{t})(0) X^{\alpha^*}_t\right), \\
		V_t &= X_t^\T \Gamma_t(0,0) X_t + 2 X_t^\T \Lambda_t(0) + \chi_t.
	}
	\ep
\end{remark}

Conventional  linear--quadratic models can also be recovered  by considering a kernel which is a weighted sum of exponentials as detailed in the following example. This will turn out to be of crucial importance in the next section. 
\begin{example}\label{E:LQn}
	We set $d=d'=m=1$ and
	\begin{align}\label{eq:knexp} 
	K^n(t) & = \; \sum_{i=1}^n c^n_i e^{-\theta_i^n t},
	\end{align}	
	for some $n \in \mathbb N$, $c_i^n \in \R, \theta^n_i \geq 0$, $i=1,\ldots,n$. This corresponds to \eqref{eq:cmmu} with $\mu(d\theta)=\sum_{i=1}^n c^n_i \delta_{\theta_i^n}(d\theta)$ and Theorem~\ref{prop:representation_of_X} gives the representation
	\bes{\label{eq:Xnsum}
		X_t^{n,\alpha}&= g^n_0(t) + \sum_{i=1}^n  c_i^n Y^{n,i,\alpha}_t,
	}
	where $Y^{n,i,\alpha}:=Y^{\alpha}(\theta^n_i)$ are such that 
	\bes{\label{eq:Yfinitedim}
		dY^{n,i,\alpha}_t &= \Big[ - \theta_i^n  Y^{n,i,\alpha}_t + \tilde{b}\Big(t,\sum_{j=1}^n  c_j^n Y^{n,j,\alpha}_t, \alpha_t \Big)\Big]dt 
		+ \tilde{\sigma}\Big(t,\sum_{j=1}^n  c_j^n Y^{n,j,\alpha}_t, \alpha_t \Big)dW_t \\
		Y^{n,i,\alpha}_0 &= 0 , \quad i=1,\ldots,n. 
	}
	Whence, the problem reduces to a conventional linear-quadratic control for the finite-dimensional controlled system $(Y^{n,i,\alpha})_{1\leq i \leq n}$. In particular, the  system of Riccati  \eqref{eq:Riccatis_monotone} reduces to a a standard one in finite-dimension. For instance the equation for $\Gamma$ reduces to the  standard $n \times n$--matrix Riccati equation
	\begin{equation} \label{eq:riccaticlassic}
	\left\{
	\begin{array}{rl} 
	\dot{\Gamma}^n_{t} &= \;   -Q^n - (B^n)^\T \Gamma^n_{t} - \Gamma^n_{t}B^n - (D^n)^\T \Gamma^n_{t} D^n  \\
	&  \;\;  + \; \big((F^n)^\T \Gamma^n_{t} D^n + (C^n)^\T \Gamma^n_{t} \big)^\T \big(N^n+ (F^n)^\T \Gamma^n_{t} F^n\big)^{-1}\big((F^n)^\T \Gamma^n_{t} D^n + (C^n)^\T \Gamma^n_{t}\big) \\
	\Gamma^n_{T} &=  0,
	\end{array}
	\right.
	\end{equation}
	where the coefficients $(B^n,C^n,D^n,F^n,N^n,Q^n)$ $\in$ $\R^{n\times n}\times \R^{n} \times \R^{n\times n}\times\R^{n}\times \R_+\times\S_+^n$ 
	are defined by 
	\begin{align}
	B^n_{i,j} &= \; B c_i^n - \theta_i^n\delta_{ij},     &  D^n_{i,j} &= \;  D c_i^n,\\
	C^n_i &= \; C c_i^n, &  F^n_i &= \;  F c_i^n, \\
	Q^n_{i,j} &= \; Q, &  N^n &= \;  N,
	\end{align}
	for all $1 \leq i,j\leq n$.
	\ep
\end{example}

\begin{remark}
	The proofs of Theorems~\ref{T:Riccatiexistence} and \ref{T:mainoptimalsolution} can be easily adapted to account for a multi-dimensional Brownian motion and time-dependent bounded coefficients.
	\ep
\end{remark}

\subsection{Stability and approximation by conventional LQ  problems}
The second main result of the paper concerns the approximation of the possibly non-Markovian control problem by sequences of finite dimensional Markovian ones, which is of crucial importance for numerical implementations.  The main idea comes from the appro\-ximation of the measure $\mu$,  appearing in \eqref{eq:cmmu}, by simpler  measures $\mu^n$, or  
equi\-valently approximating $K$ by simpler kernels $K^n$ given by 
\begin{align}\label{eq:Knmun}
K^n(t)=\int_{\R_+} e^{-\theta t} \mu^n(d\theta), \quad t >0.
\end{align} 
We also authorize the approximation of the input curve $g_0$.  By substituting $(K,g_0)$ with  $(K^n,g_0^n)$, the approximating problem  reads
\bes{
	\label{eq:approx_n}
	V^n_0=	&	\inf_{\c \in \Acal} J^n(\c) 
}
where 
\bes{
	J^n(\c) &= \;  \inf_{\c \in \Acal} \E \left[\int_0^T \left((X^{n, \alpha}_s)^{\top} Q X^{n, \alpha}_s + 2 L_s^\top X^{n, \alpha}_s + \c^\top_s N\c_s\right) ds \right], \\
	X_t{n, \alpha}&=\;  g^n_0(t) + \int_0^t K^n(t-s)b(s,X_s^{n, \alpha},\alpha_s) ds +   \int_0^t K^n(t-s)\sigma(s,X_s^{n, \alpha},\alpha_s) dW_s. \label{eq:approxprob}
}

The following theorem establishes the stability of stochastic Volterra linear--quadratic control problems. Its proof is given in Section~\ref{S:approximationproof}. 
\begin{theorem}\label{T:mainstability}  	Let $\beta,\gamma$ be bounded and measurable functions on $[0,T]$ and $g_0$ be continuous. Assume that $\mu$ satisfies \eqref{eq:totalvar}--\eqref{eq:mucont} and let $K$ be	as in \eqref{eq:cmmu}. Let  $(g_0^n)_{n \geq 1}$ be a sequence of continuous  functions and  $(K^n)_{n \geq 1}$ be a sequence of kernels of the form \eqref{eq:Knmun} with  respective   measures $\mu^n$ satisfying  \eqref{eq:totalvar}-\eqref{eq:mucont}, for each $n \in \N$. Assume \eqref{assumption:QN} and that $Q$ is invertible.
	Denote by  $V^*$ and $V^{n*}$ the respective optimal value processes given by Theorem~\ref{T:mainoptimalsolution} for the respective inputs $(g_0,K)$ and $(g_0^n, K^n)$, for $n\geq 1$. 
	If 
	\bes{
		\label{assumption:approximation}
		\|K^n - K\|_{L^2(0,T)} \to 0 \quad \mbox{and} \quad \|g_0^n - g_0\|_{L^2(0,T)} \to 0, \quad \mbox{as } n \to \infty, 
	}
	then, 
	\begin{align}\label{eq:convvaluefunction}
	V^{n*}_0 \to V^*_0, \quad \mbox{as } n \to \infty,
	\end{align}
	with a rate of convergence given by
	\begin{align} \label{convVn} 
	|V^*_{0} - V^{n*}_{0}|  & \leq c   \Big( \|g_0^n- g_0\|_{L^2(0,T)}  + \|K^n-K\|_{L^2(0,T)}   \Big),
	\end{align} 
	for some positive constant $c$ independent of $n$. 
\end{theorem}

\vspace{1mm}

Combined with Example~\ref{E:LQn}, Theorem~\ref{T:mainstability} provides an approximation of linear--quadratic stochastic Volterra optimal control problems by conventional Markovian linear--quadratic models in finite dimension. To ease notations we restrict to the case $d=d'=m=1$, for higher dimension matrices need to be replaced by  tensors in what follows. The idea is to approximate 
$\mu$ by a discrete measure $\mu^n$ as follows.   Fix $n \geq 1$ and $(\eta^n_i)_{0 \leq i \leq n}$ a partition of $\R_+$. Let 
$\mu^n(d\theta)$ $=$ $\sum_{i=1}^n c_i^n \delta_{\theta^n_i}(d\theta)$ with 
\begin{align}\label{eq:cixi}
c^n_i \; = \;  \int_{\eta^n_{i-1}}^{\eta^n_i} \mu(dx) \quad  \mbox{ and } \quad  \theta^n_i \; = \;  \frac 1 {c^n_i} \int_{\eta^n_{i-1}}^{\eta^n_i} \theta \mu(d\theta), \quad i=1,\ldots, n.
\end{align}
Then, for a suitable choice of the partition  $(\eta^n_i)_{0 \leq i \leq n}$, we obtain the convergence 
$$ \|K^n- K\|_{L^2(0,T)} \to 0, \quad \mbox{ as $n\to \infty$},$$
where $K^n$ is given by \eqref{eq:knexp},
see for instance \cite[Proposition 3.3 and Remark 3.4]{AJEE18a}. In particular,  with the fractional kernel $K_H$ given by \eqref{eq:kernelfrac}, an even $n$, and the geometric partition $\eta^n_i=r_n^{i-n/2}$, $i=0,\dots,n$, for some  $r_n>1$,
the coefficients \eqref{eq:cixi} with $\mu_H$ as in  \eqref{eq:mufrac} are explicitly given by
{\begin{align}\label{eq: ci and xi}
	c^n_i=\frac{(r_n^{1-\alpha}-1)r_n^{(\alpha-1)(1+n/2)}}{\Gamma(\alpha)\Gamma(2-\alpha)} r_n^{(1-\alpha)i} \;\;\mbox{ and }\;\;  x^n_i= \frac{1-\alpha}{2-\alpha}\frac{r_n^{2-\alpha}-1}{r_n^{1-\alpha}-1} r_n^{i-1-n/2},  \;\;  i=1,\ldots,n,
	\end{align}} 
where $\alpha:=H+1/2$.
If the sequence $(r_n)_{n \geq 1}$ satisfies
\begin{align}\label{eq:rn cond}
r_n \downarrow 1 \quad \mbox{and} \quad  n \ln r_n \to \infty, \quad \mbox{as } n\to \infty,
\end{align}
then, 
$$ \|K^n- K_H\|_{L^2(0,T)} \to 0, \quad \mbox{ as $n\to \infty$},$$ see \cite[Lemma A.3]{abi2019lifting}. 
In practice, the free  parameter $r_n$ can be chosen by minimizing the $L^2$ norm between $K^n$ and $K_H$,  for instance if $n=20$, setting $r_{20}=2.5$ yields very good approximations for the un-controlled  stochastic Volterra equation,  see \cite{abi2019lifting} for a more detailed practical study.  For each $n$, the approximate control problem is a conventional linear quadratic one in finite dimension for the state variables \eqref{eq:Yfinitedim} with the standard $n\times n$ matrix Riccati equation \eqref{eq:riccaticlassic}. This allows to numerically solve the Riccati equations and simulate the process $X^{n,\alpha}$ given by \eqref{eq:Xnsum}, leading to computation of the value function $V_0^{n*}$ and the optimal control $\alpha^n$ as in \eqref{eq:optimalcontrol_monotone} with $\mu$ replaced by $\mu^n$. 

\section{An infinite dimensional SDE with Lipschitz coefficients}\label{S:sde}

We aim to establish the existence of a solution to the stochastic Volterra equation and that of an admissible optimal control. 
For this, we shall study more generally the existence and uniqueness of a solution to an infinite dimensional stochastic differential equation (SDE) in $L^1(\mu)$. Throughout this section, we fix 
$t$ $\in$ $[0,T]$, $d,d',n\in \N$, $p \geq 2$,  $\mu$ a $d\times d'$-measure {satisfying \eqref{eq:totalvar}--\eqref{eq:mucont}},  and $W$ denotes an $n$--dimensional standard Brownian motion.

Let us  consider the infinite dimensional SDE in $L^1(\mu)$: 
\begin{align} \label{eq:eqYinfini}
d\widetilde Y_s & = \; \big( A^{mr}  \widetilde Y_s + \delta(s,\widetilde Y_s) \big) ds +  \Sigma(s,\widetilde Y_s) dW_s, \;\;\;  \quad \widetilde Y_t \; = \; \xi,
\end{align}
on $[t,T]$, where $A^{mr}$ is the mean-reverting operator as defined in \eqref{meanoperator}, the inputs $\xi$ $\in$ $L^1(\mu)$,  and 
$\delta$ $:$ $[0,T]\times\Omega\times L^1(\mu)$ $\to$  {$L^\infty(\mu)$}, 
$\Sigma$ $:$ $[0,T]\times\Omega\times L^1(\mu)$ $\to$ {$L^\infty(\mu)^n$}.

We look   for $L^1(\mu)$--valued solutions to \eqref{eq:eqYinfini} in the strong probabilistic sense and in the  mild analytical sense. More precisely, given a filtered probability space $(\Omega,\Fc, (\Fc_s )_{s \geq 0},\P )$ supporting {an} $n$ dimensional Brownian motion $W$,  we say that a progressively measurable process $\widetilde  Y$  is a (mild) solution to \eqref{eq:eqYinfini} 
on $[t,T]$ if for each $s \in [t,T]$,
\bes{
	\widetilde 	Y_s(\theta) \; = \;  e^{-\theta(s-t)} \xi(\theta)+ & \int_t^s e^{-\theta(s-u)} \delta(u,\widetilde Y_u)(\theta)  du    \\
	&\quad  +\int_t^s e^{-\theta(s-u)} \Sigma(u,\widetilde Y_u)(\theta)  dW_u  , \quad \quad  \mu-a.e.,
	\label{eq:SDE_infinite_dim}
}
such that 
\begin{align}\label{eq:estimateYLmu}
\sup_{t\leq s \leq T} \E\left[ \|\widetilde Y_s\|^p_{L^1(\mu)}  \right] < \infty.
\end{align}

The following theorem establishes the strong existence and uniqueness of a solution to \eqref{eq:SDE_infinite_dim} under Lipschitz conditions.

\begin{theorem}\label{T:existenceY}  
	Fix  $p\geq 2$ and $t \leq T$.  Assume that $\delta$ and $\Sigma$  are progressively measurable and that there exists positive constants  $c_{\rm LG}$,  $c_{\rm Lip}$, 
	and  a progressively measurable process $\phi$ with 
	$$\sup_{t \leq s \leq T}\E [|\phi_s|^p]<\infty,$$
	such that for all $y,y' \in L^1(\mu),$ and  $t\leq s \leq T$, 
	\begin{align}
	| \delta(s,y)(\theta)| +  | \Sigma (s,y)(\theta)| & \leq \;  c_{\rm LG}  \left(1 + |\phi_s| + \| y \|_{L^1(\mu)} \right),  	\label{assumption:EDS_dim_inf_growth}  \\
	| \delta(s,y)(\theta)- \delta(s,y')(\theta)|  +   |\Sigma(s,y)(\theta)- \Sigma(s,y')(\theta)| & \leq  \;  c_{\rm Lip}    \| y-y' \|_{L^1(\mu)}, 		 \label{assumption:EDS_dim_inf_lip}
	\end{align}
	$\P \otimes \mu-a.e.$  Then,  for any  $\Fc_t$--measurable random variable $\E[\|\xi\|^p_{L^1(\mu)}]<\infty$, 
	there exists a unique strong solution $Y$ to \eqref{eq:SDE_infinite_dim} on $[t,T]$ such that  \eqref{eq:estimateYLmu} holds.
\end{theorem}
\begin{proof}
	The proof is an application of the  contraction mapping principle. We denote by $\mathcal S^p_{t,T}$ the space of progressively measurable processes $\widetilde Y:\Omega\times  [t,T]\to L^1(\mu)$  such that 
	\begin{align*}
	\| \widetilde Y \|_{\Scal^p_{t,T}} &:= \;  \sup_{t\leq s \leq T} \E\left[ \|\widetilde  Y_s\|_{L^1(\mu)}^p \right]^{1/p} <\infty.
	\end{align*}
	$(\mathcal S^p_{t,T},\|\cdot\|_{\mathcal S^p_{t,T}})$ is  a Banach space. We consider the following family of norms on $\mathcal S^p_{t,T}$:
	\[
	\|\widetilde Y\|_{\lambda} :=  \sup_{t\leq s \leq T} e^{-\lambda (s-t)} \E\left[ \|\widetilde Y_{s}\|^p_{L^1(\mu)}\right]^{1/p}, \quad \lambda >0.
	\]
	For every $\widetilde Y \in \Scal^p_{t,T}$, define a new process $\Tcal \widetilde Y$  by 
	\bes{
		(\Tcal{\widetilde Y})_s(\theta) &= \; e^{-\theta(s-t)} \xi(\theta) + \int_t^s e^{-\theta(s-u)} \delta(u,\widetilde Y_u)(\theta)  du    \\
		&\quad  \quad  + \; \int_t^s e^{-\theta(s-u)} \Sigma(u,\widetilde Y_u)(\theta)  dW_u  \\
		&= \; \textbf{I}_s(\theta) +\textbf{II}_s(\theta)  + \textbf{III}_s(\theta),   \quad \quad  \mu-a.e., \quad t\leq s \leq T.
	}
	Since the norms $\|\cdot\|_{\mathcal S^p_{t,T}}$ and $\|\cdot\|_{\lambda}$ are equivalent, it is enough to find $\lambda>0$ such that $\mathcal T$ defines a contraction on $(\mathcal S^p_{t,T},\|\cdot\|_{\lambda})$. That is, we look for $\lambda>0$ and $M<1$ such that
	\begin{equation}\label{eq:contract temp_sde}
	\| \mathcal T \widetilde Y - \mathcal T \widetilde Z \|_{\lambda} \leq M  \| \widetilde Y-  \widetilde Z \|_{\lambda}, \quad  \widetilde Y,\widetilde Z \in \mathcal S^p_{t,T}.	
	\end{equation} 
	\textit{Step 1:}  	We first prove that $\Tcal (\mathcal S^p_{t,T}) \subset  \mathcal S^p_{t,T}$.  Fix $\widetilde  Y \in  \mathcal S^p_{t,T}$ and  $t\leq s \leq T $.  $ \mathcal T \widetilde  Y$ is again progressively measurable. Jensen's inequality applied  to the convex function $\|\cdot\|^p_{L^1(\mu)}$ leads to 
	\begin{align*}
	\|(\mathcal T \widetilde  Y)_s\|^p_{L^1(\mu)} \leq  3^{p-1}\left( \|\textbf{I}_s\|^p_{L^1(\mu)} + \|\textbf{II}_s\|^p_{L^1(\mu)} + \|\textbf{III}_s\|^p_{L^1(\mu)}\right).
	\end{align*}
	Since $\E[\|\xi\|^p_{L^1(\mu)}]<\infty$, we have 
	\begin{align*}
	\E\big[ \|\textbf{I}_s\|^p_{L^1(\mu)}  \big] &= \;  	 \E\left[\|e^{-(\cdot)(s-t)} \xi \|^p_{L^1(\mu)}\right]  \leq   \E\left[\| \xi \|^p_{L^1(\mu)}\right] <\infty,
	\end{align*}
	where we used the bound $e^{-\theta(s-t)} \leq 1$, since $\mu$ is supported on $\R_+$. 
	Three successive applications of Jensen's inequality on the normalized measures 
	$$\frac {(1\wedge \theta^{-1/2})\mu(d\theta)}{\int_{\R_+}  (1\wedge \tau^{-1/2})\mu(d\tau)}, \quad \frac{du}{(s-t)}, \quad  \frac{e^{-2\theta(s-u)}du}{\int_s^t e^{-2\theta(s-v)}dv},$$ yield for a constant   $c$  that may vary from line to line 
	\begin{align*}
	\|\textbf{II}_s\|^p_{L^1(\mu)} &= \;  \left(\int_{\R_+}  |\mu|(d\theta) \left| \int_t^s e^{-\theta(s-u)} \delta(u,\widetilde Y_u)(\theta)  du \right| \right)^p \\
	&\leq \; c  \int_{\R_+}  |\mu|(d\theta) \left(1\wedge \theta^{(p-1)/2}\right) \left|	\int_t^s e^{-\theta(s-u)} \delta(u,\widetilde Y_u)(\theta)  du \right|^p   \\
	&\leq \;  c  \int_{\R_+}  |\mu|(d\theta) \left(1\wedge \theta^{(p-1)/2}\right) \left|	\int_t^s e^{-2\theta(s-u)} |\delta(u,\widetilde Y_u)(\theta)|^2  du \right|^{p/2}   \\
	&\leq \; c  \int_{\R_+}  |\mu|(d\theta) \left(1\wedge \theta^{(p-1)/2}\right)\\
	&\quad \quad \quad \times \int_t^s e^{-2\theta(s-u)} |\delta(u,\widetilde Y_u)(\theta)|^p du  \left(\int_t^s e^{-2\theta(s-v)} dv\right)^{(p-2)/2}.
	\end{align*}	
	Taking expectation combined with the growth condition  \eqref{assumption:EDS_dim_inf_growth} and the fact that $\widetilde Y \in \mathcal S^p_{t,T}$ leads to 
	\begin{align*}
	\E \left[ \|\textbf{II}_s\|^p_{L^1(\mu)} \right] &\leq c\left( 1 + \sup_{t\leq u \leq T }\E \left[ |\phi_u|^p \right]  +  \|\widetilde Y\|^p_{\mathcal S^p_{t,T}}  \right) \\
	&\quad \quad \quad \quad \quad \quad \times    \int_{\R_+}  |\mu|(d\theta) \left(1\wedge \theta^{(p-1)/2}\right) 	\left(\int_t^s e^{-2\theta(s-u)} du \right)^{p/2}\\
	&=c  \int_{\R_+}  |\mu|(d\theta) \left(1\wedge \theta^{(p-1)/2}\right)\left( \frac{  1 - e^{-2\theta(s-t)}}{2\theta}\right)^{p/2}.
	\end{align*}
	
	Similarly, combining the same Jensen's inequalities with the 
	Burkholder-Davis-Gundy  ine\-quality, we get  	
	\begin{align*}
	\E\left[\|\textbf{III}_s\|^p_{L^1(\mu)}\right] &= \;  \E\left[\left(\int_{\R_+}  |\mu|(d\theta) \left|	\int_t^s e^{-\theta(s-u)} \Sigma(u,\widetilde Y_u)(\theta)  {dW_u}du \right| \right)^p \right] \\
	&\leq \; c  \int_{\R_+}  |\mu|(d\theta) (1\wedge \theta^{(p-1)/2})  \\
	&\quad \quad \quad \times \int_t^s e^{-2\theta(s-u)} \E\left[|\Sigma(u,\widetilde Y_u)(\theta)|^p\right] du \left(\int_s^t e^{-2\theta(s-v)} dv\right)^{(p-2)/2} \\
	&\leq \; c   \int_{\R_+}  |\mu|(d\theta) \left(1\wedge \theta^{(p-1)/2}\right)\left( \frac{  1 - e^{-2\theta(s-t)}}{2\theta}  \right)^{p/2}
	\end{align*}
	where the last inequality follows from the growth condition  \eqref{assumption:EDS_dim_inf_growth} and the fact that $\widetilde Y \in \mathcal S^p_{t,T}$.
	Recalling  inequality \eqref{eq:inequalityexp}, we get that 
	\begin{align}\label{eq:temp exp mu}
	\int_{\R_+}  |\mu|(d\theta) \left(1\wedge \theta^{(p-1)/2}\right) \left( \frac{  1 - e^{-2\theta(s-t)}}{2\theta}  \right)^{p/2} \leq c  \int_{\R_+}  |\mu|(d\theta) \left(1\wedge \theta^{-1/2}\right)
	\end{align} 
	which is finite due to condition \eqref{eq:totalvar}. This shows that 
	$$ \E\left[\|\textbf{II}_s\|^p_{L^1(\mu)}\right] +   \E\left[\|\textbf{III}_s\|^p_{L^1(\mu)}\right] \; \leq \;  c \; < \;  \infty.$$
	Combining the above proves that $\|\Tcal \widetilde Y \|_{\mathcal S^p_{t,T}}< \infty$ and hence $\Tcal : \mathcal S^p_{t,T} \to \mathcal S^p_{t,T}$.\\
	\textit{Step 2:} We prove that there exists $\lambda>0$ such that~\eqref{eq:contract temp_sde} holds.  Let $\widetilde  Y,\widetilde  Z \in \mathcal S^p_{t,T}$ such that $\|\widetilde  Y\|_{\lambda}$ and  $\|\widetilde  Z\|_{\lambda}$ are finite. 	Similarly to Step 1,  Jensen and Burkholder--Davis--Gundy inequalities combined with the Lipschitz condition \eqref{assumption:EDS_dim_inf_lip} lead to
	\bes{\|\mathcal T \widetilde  Y - \mathcal T \widetilde  Z\|^p_{\lambda}  \;  \leq \;   M(\lambda) \|\widetilde  Y -  \widetilde  Z\|_{\lambda}^p ,
	}
	where 
	\bes{
		M(\lambda) &=  c\int_{\R_+}  |\mu|(d\theta) \left(1\wedge \theta^{(p-1)/2}\right) 	\int_t^s e^{-2\theta(s-u)} e^{-\lambda p(s-u)}du \left(\int_{t}^s e^{-2\theta(s-v)}dv\right)^{(p-2)/2}.
	}
	By the dominated convegence theorem, recall \eqref{eq:temp exp mu},  $M(\lambda)$ tends to $0$ as $\lambda$ goes to $+\infty$.  We can therefore choose $\lambda_0 > 0$  so that  \eqref{eq:contract temp_sde} holds with $M(\lambda_0) < 1$.
	An application of the contraction mapping theorem  yields the claimed existence and uniqueness statement in $(\mathcal S^p_{t,T},\| \cdot\|_{\mathcal S^p_{t,T}})$ together with \eqref{eq:estimateYLmu}.  
\end{proof}

\begin{example}\label{ex:coefficientssde}
	Fix $\alpha \in \Acal$,  the conditions \eqref{assumption:EDS_dim_inf_growth}-\eqref{assumption:EDS_dim_inf_lip} are satisfied for the following specification  of $\delta$ and $\Sigma$:
	\begin{align*}
	\delta(s,\omega,y)(\theta) &= b_0(s,\omega,\theta) +  \int_{\R_+} B_0(s,\omega,\theta,\tau)\mu(d\tau)y(\tau) +  C_0(s,\omega,\theta)\alpha_s (\omega)  \\
	\Sigma(s,\omega,y)(\theta) &= \gamma_0(s,\omega,\theta)  +  \int_{\R_+} D_0(s,\omega,\theta,\tau)\mu(d\tau)y(\tau) +  F_0(s,\omega,\theta)\alpha_s (\omega), 
	\end{align*}
	where
	\begin{align}
	|b_0(s,\theta)| + |\gamma_0(s,\theta)|  + |C_0(s,\theta)|  + |F_0(s,\theta)| &\leq \;  c , \quad  \P \otimes \mu-a.e., \quad t\leq s\leq T, \label{eq:examplecoeff1}\\
	|B_0(s,\theta,\tau)|   + |D_0(s,\theta,\tau)| & \leq \;  c , \quad \P \otimes  \mu\otimes \mu-a.e., \quad t\leq s\leq T,  \label{eq:examplecoeff2}
	\end{align}
	for some constant $c$.
	\ep
\end{example}

\vspace{1mm}

The existence and uniqueness of a strong solution to the stochastic Volterra equation \eqref{eq:sve} readily follows from Theorem~\ref{T:existenceY} when combined with Theorem~\ref{prop:representation_of_X}. To prove continuity of the solution we need the following lemma {which makes use of condition \eqref{eq:mucont}}.

\begin{lemma}\label{L:contuinuity Y}
	Let $Z_t =\int_0^t b_s ds + \int_0^t \sigma_s dW_s$, $0\leq t\leq T$, such that $b$ and $\sigma$ are progressively measurable and 
	$$ \sup_{t\leq T}\E\left[ |b_t|^4  \right] +   \sup_{t\leq T} \E\left[ |\sigma_t|^4  \right] < \infty.  $$
	Then, the process 
	$$ \widetilde Y_t(\theta) \; = \;  \int_0^t e^{-\theta(t-s)}dZ_s, \quad \theta \in \R_+,$$
	solution in the mild sense to
	\begin{align*}
	d\tilde Y_s & = \;  A^{mr} \tilde Y_s ds + dZ_s, \;\;\; \tilde Y_0 \; = \; 0, 
	\end{align*}
	admits a  continuous  modification in $L^1(\mu)$ and satisfies
	\eqref{eq:estimateY4}-\eqref{eq:estimateY4.2}.
\end{lemma}

\begin{proof}
	The  bound  \eqref{eq:estimateY4} follows along the lines of the estimates in step 1 of the proof of Theorem~\ref{T:existenceY} with $p=4$, for getting \eqref{eq:estimateYLmu}, by  successive applications of Jensen inequalities. Let us now show \eqref{eq:estimateY4.2} and  the continuity statement.  
	Fix $\theta\in \R_+$ and $t\leq T$. An integration by parts leads to 
	$$ \widetilde Y_t(\theta) = e^{-\theta t}Z_t + \theta  \int_0^t e^{-\theta(t-s)}(Z_t-Z_s) ds.$$ 
	The Kolmogorov--Chentsov continuity criterion, yields that 	for each $\zeta \in (0,1/4)$, the process $Z$ admits a version with $\zeta$-H\"older sample paths on $[0,T]$. We identify $Z$ with this version so that
	$$  |Z_t(\omega) - Z_s(\omega)| \leq   c_{T,\zeta}(\omega) |t-s|^{\zeta}, \quad s,t\leq T,$$
	for some $c_{T,\zeta}(\omega) \geq 0$. Using this inequality and another integration by parts yields
	\begin{align*}
	| \widetilde  Y_t(\theta,\omega)| &\leq \;    c_{T,\zeta}(\omega) e^{-\theta t}t^{\zeta}+  c_{T,\zeta}(\omega) \theta \int_0^t e^{-\theta u}u^{\zeta} du\\
	&=  \; c_{T,\zeta}(\omega) \zeta \int_0^t e^{-\theta u}u^{\zeta-1} du.
	\end{align*} 
	This  proves  \eqref{eq:estimateY4.2}.  Furthermore,
	$$\sup_{t\leq T}| \widetilde  Y_t(\theta,\omega)| \leq  c_{T,\zeta}(\omega) \zeta \int_0^T e^{-\theta u}u^{\zeta-1} du,$$
	where the right hand side is in $L^1(|\mu|)$  {for some $\zeta \in (0,1/4)$ by virtue of \eqref{eq:mucont}.} Since $t\mapsto \widetilde Y_t(\theta,\omega)$ is continuous for each 
	$\theta\in \R_+$, the dominated convergence theorem yields that the process $\widetilde Y$ is continuous in $L^1(\mu)$. 
\end{proof}

\begin{theorem}\label{P:strong}  	Let $g_0$ be continuous, $\beta,\gamma$ be bounded measurable functions on $[0,T]$ and $K$ be a kernel as in  \eqref{eq:cmmu} such that  \eqref{eq:totalvar}-\eqref{eq:mucont} hold.  Fix   an admissible control  $\alpha \in \mathcal A$. 
	The stochastic Volterra equation  \eqref{eq:sve} admits a unique continuous and adapted strong solution $X^{\alpha}$   such that \eqref{eq:estimate moment X} holds.
\end{theorem}

\begin{proof} Existence, uniqueness and  \eqref{eq:estimate moment X} are straightforward from Theorem~\ref{T:existenceY} combined with Theorem~\ref{prop:representation_of_X} and Example~\ref{ex:coefficientssde} for the coefficients \begin{align*}
	b_0(s,\theta)&= \;  \beta(s) + Bg_0(s),  \quad  B_0(s,\theta,\tau)= B,  \quad C_0(s,\theta,\tau)=C,\\
	\gamma_0(s,\theta)&= \;  \gamma(s) + Dg_0(s),   \quad D_0(s,\theta,\tau)=D,   \quad F_0(s,\theta,\tau)=F.
	\end{align*}
	The statement concerning the continuity of $X^{\alpha}$ is a direct consequence of the continuity of $Y^{\alpha}$  established in  Lemma~\ref{L:contuinuity Y} and the converse  direction in Theorem~\ref{prop:representation_of_X}.
\end{proof}

\section{A martingale verification theorem}\label{S:verification}

We first derive an It\^o formula for quadratic functions in $L^1(\mu)$, where we fix  a $d\times d'$-matrix measure $\mu$ on $\R_+$ such that   $|\mu|$ is   $\sigma$-finite.

For fixed $t$ $\in$ $[0,T]$,  let  $\widetilde \Gamma,\widetilde \Lambda \in  C([t,T],L^1(\mu\otimes \mu))\times C([t,T],L^1( \mu^\top))$ be solutions to 
\begin{align}
\dot   {{\widetilde\Gamma}}_{s}(\theta,\tau) &=\;  (\theta+\tau)\tilde\Gamma_{s}(\theta,\tau) -  R^1_s(\theta,\tau), &\quad   t\leq s \leq T,   \quad  \mu\otimes \mu-a.e. \\
\dot   {{\widetilde\Lambda}}_{s}(\theta) &=\;   \theta \widetilde\Lambda_s(\theta)  - R^2_s(\theta), &  t\leq s \leq T,   \quad  \mu-a.e. \label{eq:tildeLambda}
\end{align}
{with $\tilde\Gamma_T$ $\in$ $L^\infty(\mu\otimes\mu)$, $\widetilde\Lambda_T$ $\in$ $L^\infty(\mu^\top)$},  and 
for some measurable functions $s$ $\mapsto$ $R_s^1$ and $R_s^2$ valued respectively in $L^{\infty}(\mu\otimes\mu)$ and  $L^{\infty}(\mu^\top)$, such that 
\begin{align} \label{integR}
\int_t^T  
\|R^1_s\|_{L^\infty(\mu\otimes\mu)} ds + \int_t^T \|R^2_s\|_{L^\infty(\mu^\top)}ds 
& < \infty.    
\end{align}
Recall  that $\widetilde\Gamma_s$ $\in$ $L^\infty(\mu\otimes\mu)$, $\widetilde\Lambda_s$ $\in$ $L^\infty(\mu^\top)$, for all $t\leq s\leq T$.

\begin{lemma} \label{L:ito} 
	Fix $t$ $\in$ $[0,T]$, and let $\widetilde Y$  be a $L^1(\mu)$-valued progressively measurable processes  solution in the mild sense to
	\begin{align} \label{eq:tildeYmild}
	d\tilde Y_s & = \;  A^{mr} \tilde Y_sds +  b_s ds + \sigma_s dW_s, \quad t \leq s \leq T,  \quad  \tilde Y_t = \xi,
	\end{align}  
	for  $\xi\in L^1(\mu)$ and some progressively measurable $b,\sigma$ valued in $L^\infty(\mu)$ and satisfying 
	\begin{align} \label{integbsig} 
	\int_t^T \| b_s \|_{L^\infty(\mu)} ds + \int_t^T \|\sigma_s\|_{L^\infty(\mu)}^2 ds & < \;\infty, \quad  \P-a.s. 
	\end{align}
	Assume that $\widetilde Y$ is {bounded in $s$ $\in$ $[t,T]$,  $\P\otimes\mu$-a.e,   and has continuous sample paths in $L^1(\mu)$.}  
	Then, given $\tilde\Gamma$, $\tilde\Lambda$ as in   \eqref{eq:tildeLambda}, the processes
	\begin{align*}
	U_s^1 & =  \langle \widetilde Y_s,  \boldsymbol{\widetilde\Gamma_s} \widetilde Y_s \rangle_\mu {\; = \;  \langle  \boldsymbol{\widetilde\Gamma_s} \widetilde Y_s, \widetilde Y_s \rangle_{\mu^\top} 
		\quad\quad  U_s^2 \; = \;   \langle  \tilde Y_s , \tilde\Lambda_s\rangle_{\mu} \: = \; 
		\langle \widetilde\Lambda_s,\widetilde Y_s \rangle_{\mu^\top} }, 
	\quad t \leq s \leq T, 
	\end{align*}
	are well defined, where $\boldsymbol{\widetilde\Gamma}$ is the integral operator associated to the kernel $\widetilde\Gamma$. 
	Furthermore, we have for $i$ $=$ $1,2$, 
	\begin{align}  \label{eq:ItoU2}
	dU_s^i & = \;  \Delta_s^i ds +  \Sigma_s^i dW_s, \quad \quad t \leq s \leq T, 
	\end{align}
	where 
	\begin{align}
	\Delta_s^1 & =  - \langle \widetilde Y_s , \boldsymbol{R_s^1} \widetilde Y_s\rangle_\mu  +  \langle \boldsymbol{\widetilde\Gamma_s} \sigma_s, \sigma_s  \rangle_{\mu^\top} 
	+  \langle \boldsymbol{\widetilde\Gamma_s} \widetilde Y_s, b_s  \rangle_{\mu^\top}  + \langle \boldsymbol{\widetilde\Gamma_s} b_s, \widetilde Y_s  \rangle_{\mu^\top}, \\
	\Sigma_s^1  &=   \langle \boldsymbol{\widetilde\Gamma_s} \widetilde Y_s, \sigma_s  \rangle_{\mu^\top} +   
	\langle \boldsymbol{\widetilde\Gamma_s} \sigma_s, \widetilde Y_s  \rangle_{\mu^\top}, \\
	\Delta_s^2 & = -   \langle \widetilde Y_s, R_s^2   \rangle_{\mu}   + \langle  \widetilde \Lambda_s, b_s   \rangle_{\mu^\top}, \quad \quad \Sigma_s^2  
	\;  =  \langle  \widetilde\Lambda_s  , \sigma_s \rangle_{\mu^\top}  ,
	\end{align}
	where $\boldsymbol{R^1}$ is the integral operator associated to the kernel $R^1$. 
\end{lemma}

\begin{proof} We illustrate the proof only for $U^2$, that of $U^1$ follows along the same lines.  The idea is to apply It\^o's formula {for each $\theta$}.   By virtue of the inequality $|\mu(B)| \leq |\mu|(B)$, for any Borel set $B$, and  the $\sigma$-finiteness of $|\mu|$, an application of the Radon--Nikodym theorem yields the existence of a measurable function $h:\R_+ \to \R^{d\times d'}$  such that 
	\begin{align}\label{eq:RadonN}
	\mu(d\theta) &= h(\theta) |\mu|(d\theta)
	\end{align}
	with  $|h(\theta)|=1$, for all $\theta\in \R^d$,  and $ |h| \in L^1(|\mu|)$, see for instance \cite[Lemma 3.5.9]{GLS:90}. Recall that by definition of a mild solution to 
	\eqref{eq:tildeYmild}, we mean that 
	\begin{align}
	\widetilde Y_s(\theta) &=  e^{-\theta(s-t)}\xi (\theta) +\int_t^s e^{-\theta(s-u)} b_u (\theta) du \\
	&\quad +\int_t^s e^{-\theta(s-u)}\sigma_u(\theta) dW_u,  \quad t \leq s \leq T, \quad \P\otimes \mu-a.e. \label{eq:tildeY}
	\end{align}
	Fix $t\leq T $,  and observe that the solution   $\widetilde\Lambda$ to \eqref{eq:tildeLambda} is given by 
	\begin{align*}
	\widetilde\Lambda_s(\theta) &= e^{-\theta(T-s)} \widetilde\Lambda_T(\theta) + \int_s^T e^{-\theta(u-t)} R_u^2(\theta) du, \quad t \leq s\leq T, \quad  \mu-a.e.,
	\end{align*} 
	which implies, with \eqref{integR}, that $\widetilde\Lambda_s$ $\in$ $L^\infty(\mu^\top)$, $t\leq s\leq T$, together with the boundedness of $s\mapsto\widetilde\Lambda_s$: 
	\begin{align} \label{Lambdabor} 
	\sup_{t\leq s\leq T} \big|\widetilde\Lambda_s(\theta)\big|  & \leq \;  \big\|\widetilde\Lambda_T\big\|_{L^\infty(\mu^\top)}     +    \int_t^T  \|R_u^2\|_{L^\infty(\mu^\top)}   du \; < \; \infty, 
	\quad  \mu-a.e. 
	\end{align}
	Moreover, since $\widetilde Y$ is bounded in $s$, we have 
	\begin{align} \label{Ybor}
	\sup_{t\leq s\leq T} |\widetilde Y_s(\theta)| & <  \infty, \quad \P\otimes\mu-a.e.
	\end{align}
	Define the $\P\otimes \mu$-nullset 
	$$\mathcal N=\{(\omega,\theta): \mbox{such that  either  \eqref{eq:tildeY} or  \eqref{eq:tildeLambda} or \eqref{Lambdabor} or \eqref{Ybor} does not hold} \}.$$
	Let $(\omega,\theta) \in \Omega\times \R_+ \setminus \mathcal N$ and observe that $s\mapsto \widetilde Y_s(\theta,\omega)$ and 
	$s\mapsto \widetilde \Lambda_s(\theta)$ solve:
	\begin{align*}
	d\widetilde Y_s(\theta,\omega) &= \left(- \theta \widetilde Y_s(\theta,\omega) +  b_s(\theta,\omega)\right)ds + \sigma_s(\theta,\omega)dW_s,\\
	d\widetilde\Lambda_s(\theta) &=  \left( \theta  \widetilde\Lambda_s(\theta) - R^2_s(\theta)  \right) ds. 
	\end{align*} 
	An application of It\^o's formula to the process $u^2(\theta,\omega):s\mapsto \widetilde \Lambda_s(\theta)^\top h(\theta)\widetilde Y_s(\theta)$ gives
	\begin{align}\label{eq:eq:tempIto}
	u^2_s(\theta,\omega) &= u^2_t(\theta,\omega) + \int_t^s \delta^2_u(\theta,\omega) du  + \int_t^s \widetilde \Lambda_u(\theta)^\top h(\theta) 
	\sigma_u(\theta,\omega) dW_u,\quad t\leq s, 
	\end{align} 
	with 
	$$ \delta_u^2(\theta,\omega) \; = \; -R^2_u(\theta)^\top h(\theta) \widetilde Y_u(\theta,\omega)  + \widetilde \Lambda_{u}(\theta)^\top h(\theta)b_u(\theta,\omega).$$
	All the quantities appearing on the right hand side of \eqref{eq:eq:tempIto} are well-defined thanks to the integrability assumptions \eqref{integbsig}-\eqref{integR} on the coefficients 
	$(b,\sigma,R^2)$ and the boundedness in $s$ of  $(\widetilde \Lambda_s,\widetilde Y_s)$ from \eqref{Lambdabor}-\eqref{Ybor}.
	Whence, \eqref{eq:eq:tempIto} holds $\P \otimes \mu$ almost everywhere.
	{Next, by the boundedness (resul\-ting from the continuity) of $s$ $\mapsto$ $\widetilde Y_s$ in $L^1(\mu)$, $s$ $\mapsto$ $\widetilde\Lambda_s$ in $L^1(\mu^\top)$,  
		and again by the integrability conditions \eqref{integbsig}-\eqref{integR}  on  $(b,\sigma,R^2)$}, all the terms appearing in \eqref{eq:eq:tempIto} are in $L^1(|\mu|)$ so that an integration with respect to the $\theta$ variable against $|\mu|$ combined with the identity  \eqref{eq:RadonN} and the stochastic Fubini's theorem, see \cite[Theorem~2.2]{V:12}, lead to \eqref{eq:ItoU2}.
\end{proof}

The next theorem establishes a martingale verification result for the possibly non-Markovian optimization problem  \eqref{eq:sve}-\eqref{eq:main problem}.
\begin{theorem}
	\label{thm:verification}
	Let $\beta,\gamma$ be bounded functions on $[0,T]$, $g_0$ continuous. Fix  $K,\mu$ as in  \eqref{eq:cmmu}-\eqref{eq:totalvar}-\eqref{eq:mucont}.
	Assume that:
	\begin{enumerate}
		\item \label{verifi}
		there exists a solution $(\Gamma,\Lambda,\chi) \in C([0,T],L^1(\mu\otimes \mu))\times C([0,T],L^1( \mu^\top)) \times C([0,T],\R)$ such that \eqref{eq:Riccati_monotone_kernel_mild}, \eqref{eq:BSDE_monotone_kernel_mild}, \eqref{eq:ODE_monotone_kernel_mild}   hold,  and {$\hat N(\Gamma_t) \in \S^m_+$, recall \eqref{eq:hatN}}, for all $t\leq T$, together with the estimate 
		\eqref{eq:estimateRiccati},
		\item  \label{verifiii}
		there exists an admissible  control process  $\alpha^* \in \mathcal A$ such that  \eqref{eq:optimalcontrol_monotone} holds for all $t\leq T$.
	\end{enumerate}
	Then, $\alpha^*$ is an optimal control,  $Y^{\alpha^*}$ given by \eqref{eq:yalpha} is the optimally controlled trajectory of the state variable and   $V^{\alpha^*}$  given by \eqref{eq:valuefunY} is the optimal value process  of the problem, in the sense that  	\eqref{eq:optimalvalue_monotone} holds, for all $t\leq T$. 
\end{theorem}

\begin{proof}
	\textit{Step 1.} For any  $\alpha \in \Acal$,  we know from Theorem  \ref{P:strong} that there exists a continuous  solution $X^{\alpha}$ to \eqref{eq:sve} such that   \eqref{eq:estimate moment X} holds. Let us then  define the  continuous process
	\begin{align*}
	M^{\alpha}_t &= 	\int_0^t {f(X_s^{\alpha},\alpha_s)}ds + V_t^{\alpha} -\int_0^t  (\c_s - {\mathcal T_s(\alpha)})^\T \Hat{N}(\Gamma_s) (\c_s - {\mathcal T_s(\alpha)}) d s, \\
	{\mathcal T_s(\alpha)}&= {-\hat{N}(\Gamma_{s})^{-1} \Big( h(s,\Gamma_{s}, \Lambda_{s}) + \int_{\R_+} S(\Gamma_{s})(\theta) \mu(d\theta) Y^{\alpha}_s(\theta)  \Big)},
	\end{align*}
	where 
	$f$ is the quadratic function in \eqref{fquadra}, $V^\alpha$ is the process given by \eqref{eq:valuefunY}  from the solution $(\Gamma,\Lambda,\chi)$ to the Riccati equation  \eqref{eq:Riccatis_monotone}, and $\hat N$ by \eqref{eq:hatN}, recall Remark~\ref{R:continuityV}. 
	The main point is to check that $M^{\alpha}$ is martingale for any $\alpha$ $\in$ $\Acal$. 
	Indeed, if this is the case, then, for each $t\leq T$, the equality 
	$\E[M^{\alpha}_{T}| \Fc_t]=M^{\alpha}_{t}$ leads to
	\begin{align}\label{eq:tempalpha}
	J_t(\alpha)  - V_t^\alpha & = \E\left[ \int_t^{T}  (\c_s - {\mathcal T_s(\alpha)})^\T \Hat{N}(\Gamma_s) (\c_s - {\mathcal T_s(\alpha)} ) d s  \Mid \cali{F}_t \right],
	\end{align}
	where we have set $J_t(\alpha)$ $=$ $\E\left[\int_{t}^{T} f(X_s^{\alpha},\alpha_s)ds \Mid \cali{F}_t \right]$, and used $V^{\alpha}_T=0$, due to the vanishing terminal conditions of $(\Gamma,\Lambda,\Theta)$ and the  continuity of $V^{\alpha}$. 
	{Since $\hat N(\Gamma_s) \in \S^m_+$, the right hand} side of \eqref{eq:tempalpha} is always nonnegative and {vanishes for $\alpha=\alpha^*$, where $\alpha^*$ is given by \eqref{eq:optimalcontrol_monotone}}.  Fix now $t \leq T$,  and observe  that $V_t^{\alpha^*}$ $=$ $V_t^{\alpha'}$ for all 
	$\alpha'$ $\in$ $\Acal_t(\alpha^*)$. We then deduce from \eqref{eq:tempalpha}  that
	\begin{align*}
	V_t^{\alpha^*} & = J_t(\alpha^*) \; = \; \inf_{\alpha'\in\Acal_t(\alpha^*)} J_t(\alpha'), 
	\end{align*}
	which is the  relation \eqref{eq:optimalvalue_monotone}, and  shows that $\alpha^*$ is an optimal control. 
	
	\vspace{1mm}

	\textit{Step 2.} We now prove that $M^{\alpha}$ is indeed a martingale by means of It\^o's formula. To ease notations, we  drop   the superscript $\alpha$  from $X^{\alpha}$ and $Y^{\alpha}$. The process $V^\alpha$ is written as 
	\begin{align*}
	V^{\alpha}_t &= U_t^1 +  2 U_t^2 + \chi_{t}, \quad 0 \leq t \leq T, 
	\end{align*}
	where
	\begin{align*}
	U_t^1 & =  \langle  Y_t,  \boldsymbol{\Gamma_t}  Y_t  \rangle_\mu, \quad\quad U_t^2 \; = \;   \langle \Lambda_t,  Y_t \rangle_{\mu^\top}, 
	\end{align*}   
	and we recall that $Y$ is bounded in $(t,\theta)$ from \eqref{eq:estimateY4.2},  and has continuous sample path in $L^1(\mu)$ by Theorem \ref{prop:representation_of_X}.    
	From \eqref{eq:estimateY4.2}, and the admissibility condition on $\alpha$ $\in$ ${\cal A}$ combined with Cauchy-Schwarz inequality, it is clear that the drift 
	$b_t$ $:=$ $\tilde \beta(t) + \boldsymbol{B}Y_t + C\alpha_t$, and the 
	diffusion coefficient  $\sigma_t$ $:=$ $\tilde\gamma(t)+ \boldsymbol{D}Y_t + F\alpha_t$ of $Y$  take values in $L^\infty(\mu)$, and satisfy the integrability condition \eqref{integbsig}.  Moreover, from  \eqref{eq:estimateRiccati},  
	and since $t$ $\in$ $[0,T]$ $\mapsto$ $\Gamma_t$, $\Lambda_t$  are bounded (by continuity) in $L^1(\mu\otimes\mu)$ and   $L^1(\mu^\top)$,  
	we see that $t$ $\mapsto$ $R_t^1$ $:=$ ${\cal R}_1(\Gamma_t)$ is valued in   $L^\infty(\mu\otimes\mu)$, $t$ $\mapsto$ $R_t^2$ $:=$ ${\cal R}_2(t,\Gamma_t,\Lambda_t)$ is valued in $L^\infty(\mu^\top)$,  and satisfy  (see Remark \ref{R:Linfty}):  
	\begin{align*}
	\sup_{t\leq T} \big[ \|R^1_t\|_{L^\infty(\mu\otimes\mu)}  + \|R^2_t\|_{L^\infty(\mu^\top)} \big] & < \; \infty, 
	\end{align*}
	which clearly implies the integrability condition \eqref{integR}.  We can then apply  Lemma~\ref{L:ito} on $U^1$, $U^2$. Recalling that $\Gamma$ is symmetric, 
	this yields, after re-arranging the quadratic and linear terms in $Y$ and $\alpha$, using the equation for $\Gamma$  in \eqref{eq:Riccatis_monotone} and applying Fubini's theorem:
	\begin{align}
	dU^1_t 
	=&\int_{\R_+^2} Y_t(\theta)^\T\mu(d\theta)^\T \Big( S(\Gamma_t)(\theta) \hat{N}(\Gamma_t)^{-1} S(\Gamma_t)(\tau) - Q \Big) \mu(d\tau)Y_t(\tau) dt \\
	& \;  + \int_{\R_+}  2 Y_t(\theta)^\T \mu(d\theta)^\T  \left(\ints \Gamma_t(\theta,\tau')\mu(d\tau') \tilde{\beta}(t)+ D^\T \int_{\R_+^2} {\mu(d\theta')^\T} 
	\Gamma_t(\theta',\tau')\mu(d\tau')\tilde{\gamma}(t)  \right)  dt\\
	&\; +   \c_t^\T F^\T \int_{\R_+^2} \mu(d \theta)^\T \Gamma_t(\theta,\tau)   \mu(d\tau)  F \c_t  dt\\
	& \; + 2 \c_t^\T \left( \int_{\R_+} S(\Gamma_t)(\tau) \mu(d\tau)Y_t(\tau) + F^\T \int_{\R_+^2} \mu(d\tau)^\T  \Gamma_t(\theta,\tau)  \mu(d\tau)\tilde{\gamma}(t)  \right) dt \\
	& \; + \bigg( { \tilde{\gamma}(t)^\T} \int_{\R_+^2} \mu(d\theta')^\T \Gamma_t(\theta',\tau') \mu(d\tau'){ \tilde{\gamma}(t)} \bigg) dt    \;  + \;  H^1_t dW_t,
	\end{align}
	with 
	\bes{
		H^1_t = 2 \sigma(t, X_t, \c_t)^\T \int_{\R_+^2}  \mu(d\theta)^\T \Gamma_t(\theta,\tau)\mu(d\tau) Y_t(\tau).
	}
	Similarly, using  the equation for $\Lambda$  in \eqref{eq:Riccatis_monotone} we get
	\bes{
		dU^2_ t =&\int_{\R_+}Y_t(\theta)^\T \mu(d\theta)^\T \Big( S(\Gamma_t)(\theta)^\T \hat{N}(\Gamma_t)^{-1}  h(t,\Gamma_t, \Lambda_t) - L - {Q g_0(t)} \\
		&  {-  \ints \Gamma_t(\theta,\tau')\mu(d\tau') \tilde{\beta}(t)- D^\T \int_{\R_+^2} {\mu(d\theta')^\T} 
			\Gamma_t(\theta',\tau')\mu(d\tau')\tilde{\gamma}(t)}  \Big )dt \\
		& \; +   \c_t^\T \bigg( C^\T  \ints \mu(d\theta')^\T \Lambda_t(\theta') \bigg) dt + \tilde{\beta}(t)^\T \ints \mu(d\theta')^\T \Lambda_t(\theta') dt \; +  \;  H^2_t dW_t,
	}
	where 
	\bes{
		H^2_t &=  {\sigma}(t, X_t, \c_t)^\T \int_{\R_+} \mu(d\theta)^\T \Lambda_t(\theta).
	}
	Now we write that 
	\begin{align*}
	dM^{\alpha}_t &= \left(X_t^{\top} Q X_t  + \alpha_t^\top N \alpha_t - (\alpha_t-{\mathcal T_t(\alpha)})^\top \hat N(\Gamma_t) (\alpha_t-{\mathcal T_t(\alpha)}) 
	+ {2} L^\top X_t + \dot \chi_t \right) dt 
	+ dU_t^1 +  2 dU^2_t. 
	\end{align*}
	{After a completion of the squares for the terms in  $\alpha$, we observe that} 
	\begin{align*}
	X_t^\top Q X_t &= \int_{\R_+^2} Y_t(\theta)^\top \mu(d\theta)^\top Q \mu(d\tau) Y_t(\tau) + 2 g_0(t)^\top \int_{\R_+}{Q} \mu(d\theta)Y_t(\theta) + g_0(t)^\top Q g_0(t), \\
	L^\top X_t &= L^\top   \int_{\R_+} \mu(d\theta)Y_t(\theta) + L^\top g_0(t),
	\end{align*} 
	using the equation  for $\chi$  in \eqref{eq:Riccatis_monotone}, and 
	adding all the above makes the drift in $M^{\alpha}$ vanish so that  
	$$  dM^{\alpha}_t = \left(H^1_t +  2 H^2_t\right) dW_t.$$
	This shows that $M^{\alpha}$ is  a local martingale. To argue true martingality, successive applications of Jensen and Cauchy--Schwarz inequalities combined with the bound \eqref{eq:estimateRiccati} yield, for a constant $c$,
	\begin{align*}
	\int_0^T \E \left[ |H^1_s|^2 \right] ds &\leq cM^2  \int_0^T \E \left[ \left(1 + |X_s|^2 + |\alpha_s|^2 \right) \|Y_s\|^2_{L^1(\mu)} \right] ds \\
	&\leq cM^2\int_0^T \E \left[ \left(1 + |X_s|^4 + |\alpha_s|^4\right) \right]^{1/2} \E\left[ \|Y_s\|^4_{L^1(\mu)} \right]^{1/2} ds 
	\end{align*}
	which is finite due to \eqref{eq:estimate moment X}, \eqref{eq:estimateY4} and the admissibility of $\alpha$.   Since $\Lambda \in C([0,T],L^1(\mu^\top))$, 
	we  get  similarly  that
	$$ \int_0^T \E \left[ |H^2_s|^2 \right] ds < \infty. $$
	Whence, by the Burkholder-Davis-Gundy inequality, $M^{\alpha}$ is a true martingale, and  the proof is complete.
\end{proof}

\begin{remark}
	Theorem~\ref{thm:verification} is still valid if one adds a linear quadratic terminal cost to the cost functional    \eqref{eq:original_problem}, that is 
	\begin{align}
	J(\alpha)  &= \E\left[\int_0^T  f(X^{\alpha}_s,\alpha_s) d s + (X^{\alpha}_T)^\top P X^{\alpha}_T + 2 U^\top X^{\alpha}_T  \right],
	\end{align}
	provided the terminal conditions of the system of Riccati equations \eqref{eq:Riccatis_monotone} are updated to 
	$$ \Gamma_{T}(\theta,\tau)= P , \quad \Lambda_{T} (\theta)=U + P g_0(T) , \quad \chi_{T}=g_0(T)^\T P g_0(T) + 2 U^\T g_0(T). $$
	The main technical difficulty resides in  Assumption \ref{verifi}. If $K$ has no singularities at $0$,  then one can still construct continuous solutions  to   \eqref{eq:Riccatis_monotone}. However, in the presence of a singularity,  the solution $t\to\Gamma_{t}$ inherits the singularity of the kernel and is no longer continuous but only lies in 
	$L^1([0,T],L^1(\mu\otimes \mu))$.
	\ep
\end{remark}

\section{Proof of solvability result}\label{S:proof}

\begin{proof}[Proof of Theorem \ref{T:Riccatiexistence}] 
	
	First,	the existence  and uniqueness of a solution   $\Gamma \in C([0,T], L^1(\mu \otimes \mu))$ to the Riccati equation \eqref{eq:Riccati_monotone_kernel_mild} satisfying the estimate 
	\eqref{eq:estimateRiccati} and such that $\Gamma_t \in \S^d_+(\mu\otimes \mu)$, for all $t\leq T$, follow from  \cite[Theorem 2.3]{abietal19b}. 
	
	Second, we note that equation \eqref{eq:BSDE_monotone_kernel_mild} for $\Lambda$ is a Lyapunov equation of the form 
	\begin{align}
	\Psi_t(\theta, \tau) &=  \int_t^T e^{-(\theta + \tau)(s-t)} F(s,\Psi_s)(\theta,\tau)ds, \quad { t\leq T, \quad \mu\otimes\mu-a.e.}	\label{def:infinite_dim_lyapunov1} \\
	F(s,\Psi)(\theta, \tau)&= \tilde Q_s(\theta, \tau) + \tilde D^1_s(\theta)^\T \int_{\R_+^2} { \mu_1(d\theta')^\T \Psi(\theta',\tau ') \mu_2(d\tau ')} \tilde D^2_s(\tau) \\
	&\quad  + \tilde B^1_s(\theta)^\T \int_{\R_+} \mu_1(d\theta')^\T \Psi(\theta',\tau)  + \int_{\R_+} \Psi(\theta,\tau ') \mu_2(d\tau ') \tilde B^2_s(\tau), 	\label{def:infinite_dim_lyapunov2}
	\end{align}
	where  $\mu_i$, $i$ $=$ $1,2$,  $d_{i1} \times d_{i2}$-matrix valued measures on $\R_+$, with
	$$ d_{11}=d, \quad d_{12}=d', \quad d_{21}=d_{22}=1, \quad \mu_1=\mu, \quad \mu_2=0, $$
	and  coefficients 
	\begin{align*}
	\tilde Q_t(\theta,\tau) &=L+Qg_0(t) +  \int_{\R_+} \Gamma_{t}(\theta,\tau')\mu(d\tau')\widetilde \beta_t \\
	&\quad -S(\Gamma_{t})(\theta)^\top \hat N(\Gamma_{t})^{-1}F^\top \int_{\R_+^2} \mu(d\theta')^\top \Gamma_{t}(\theta',\tau')\mu(d\tau') \widetilde \gamma_t,\\
	\tilde B^1_t(\theta)  &= B - C  \hat N(\Gamma_{t})^{-1} S(\Gamma_{t})(\theta) , \quad   \quad  \tilde B^2_t(\theta)= \tilde D^1_t(\theta)=\tilde D^2_t(\theta)=0.
	\end{align*}
	From \cite[Theorem 3.1]{abietal19b}, we then get the existence and uniqueness of a solution  $\Lambda \in C([0,T],L^1(\mu^\top))$ to the equation \eqref{eq:BSDE_monotone_kernel_mild}.
	Finally, we notice that the right hand side of  equation  \eqref{eq:ODE_monotone_kernel_mild}  for  $\chi$, does not involve $\chi$. Therefore, the existence and the continuity of $\chi$ follow upon simple integration, which is justified by the boundedness of $g_0,\beta,\gamma$, that  of $\Lambda$  
	in $L^1(\mu^\top)$ and that of $\Gamma$ in $L^1(\mu \otimes \mu)$ together with the bound \eqref{eq:estimateRiccati}.
\end{proof}

\vspace{1mm}

\begin{proof}[Proof of Theorem~\ref{T:mainoptimalsolution}]
	
	The result is a direct consequence of Theorem~\ref{thm:verification} once we prove that conditions \ref{verifi}-\ref{verifiii} are satisfied. Condition \ref{verifi}  follows from 
	Theorem~\ref{T:Riccatiexistence}, {and $\hat N(\Gamma)$ is $\S^m_+$-valued since  $N\in \S^m_+$ and $\Gamma$ is $\S^d_+(\mu \otimes \mu)$ valued}.  It remains to prove  condition \ref{verifiii}, i.~e. that there exists a progressively measurable  process $\alpha^*$ $\in$ $\Acal$ associated to 
	a controlled SDE $Y^{\alpha^*}$ $\in$ $L^1(\mu)$  such that \eqref{eq:optimalcontrol_monotone} holds, and $\sup_{t\leq T}\E\left[ |\alpha^*_t|^4\right]<\infty$.  
	To this end, we consider  the  coefficients $\delta,\Sigma$ as in Example~\eqref{ex:coefficientssde} with 
	\begin{align*}
	b_0(t,\theta)&=  \widetilde \beta(t)  -C  \hat{N}(\Gamma_{t})^{-1} h(t,\Gamma_{t}, \Lambda_{t})  \quad  \\
	\gamma_0(t,\theta)&= \widetilde \gamma(t)  -F \hat{N}(\Gamma_{t})^{-1} h(t,\Gamma_{t}, \Lambda_{t})\\
	B_0(t,\theta,\theta') &=	B   - C \hat{N}(\Gamma_{t})^{-1}   S(\Gamma_{t})(\theta')    \\
	D_0(t,\theta,\theta')&=	D  - F \hat{N}(\Gamma_{t})^{-1}  S(\Gamma_{t})(\theta')    \\
	C_0(t,\theta)& = F_0(t,\theta)=0.
	\end{align*}
	By the boundedness of $\beta,\gamma,g_0,$ on $[0,T]$, the continuity hence the boundedness on $[0,T]$ of  $\Gamma$ in  $L^1(\mu \otimes \mu)$ and $\Lambda$ in 
	$L^1(\mu^\top)$ together with  the bound in Theorem  \ref{T:Riccatiexistence}, the previous coefficients satisfy \eqref{eq:examplecoeff1}-\eqref{eq:examplecoeff2}.
	Whence, for any $p\geq 2$, Example~\ref{ex:coefficientssde}, combined with Theorem~\ref{T:existenceY}, yields the existence of a process $Y^*$ with initial condition $Y_0^*\equiv 0$,   for the coefficients $\delta,\Sigma$ as defined above  and such that \eqref{eq:estimateYLmu} holds.   One can therefore define a process $\alpha^*$ by 
	\begin{align}
	\c^*_t &= {-\hat{N}(\Gamma_{t})^{-1} \Bigg( h(t,\Gamma_{t}, \Lambda_{t}) + \int_{\R_+} S(\Gamma_{t})(\theta) \mu(d\theta) Y_t^*(\theta)  \Bigg)},
	\end{align}
	and see, again from the boundedness on $[0,T]$ of  $(\Gamma,\Lambda)$ in  $L^1(\mu \otimes \mu)\times L^1(\mu^\top)$ together with  the bound \eqref{eq:estimateRiccati}, that 
	\begin{align*}
	\E\big[|\c^*_t|^p\big] &\leq c(1+M^4) \sup_{0\leq t\leq T} \E\big[ \|Y_t^*\|^p_{L^1(\mu)} \big], \quad 0 \leq t \leq T,
	\end{align*}	
	which is finite due to \eqref{eq:estimateYLmu}. In particular, for $p=4$, we get that $\alpha^* $ lies in $\Acal$. \\
	Finally,  by construction, the coefficients of $Y^*$ can be re-written in terms of $\alpha^*$ as 
	\begin{align*}
	\delta(t,\omega,Y_t^*)&= \widetilde \beta(t) +  B\int_{\R_+} \mu(d\theta)  Y_t^*(\theta) + C \alpha^*_t, \\
	\Sigma(t,\omega,Y_t^*)&= \widetilde \gamma(t)  +  D\int_{\R_+} \mu(d\theta)  Y_t^*(\theta) + F \alpha^*_t,
	\end{align*}
	which means that $Y^*$ $=$ $Y^{\alpha^*}$, and ends the proof. 
\end{proof}

\section{Stability and approximation} \label{secapprox}

In this section,  we prove Theorem~\ref{T:mainstability} starting with {\it a priori}  $L^2$--estimates for the controlled stochastic Volterra equation before approximating  the value function. 

\subsection{A-priori $L^2$--estimates for stochastic Volterra equations}
Let $X^{\alpha}$ be the solution produced by Theorem~\ref{P:strong}.  
We  provide an explicit bound for $\E \left[\|X^{\alpha}\|^2_{L^2(0,T)}\right] $, which is finite due to  \eqref{eq:estimate moment X}. For this, let us introduce 
the resolvent of the second kind $R$ of a scalar kernel $k$, defined as  the unique $L^1([0,T],\R)$ solution to the linear convolution equation
\begin{align*}
R(t) = k(t) +    \int_0^t k(t-s)R(s)ds = k(t) +    \int_0^t R(t-s) k(s)ds, \quad t \leq T.
\end{align*}
Recall that the resolvent $R$ exists, for any kernel $k \in L^1([0,T],\R)$, see \cite[Theorems 2.3.1  and 2.3.5]{GLS:90}. 

\begin{lemma}
	\label{L:aprioriL2} 
	Fix $K\in L^2([0,T],\R^{d\times d'})$, $\alpha \in \Acal$,    $g_0 \in L^2([0,T],\R^d)$ 	and $\beta, \gamma \in L^2([0,T],\R^{d'})$. If  $X^{\alpha}$ is a  progressively measurable  process satisfying   \eqref{eq:sve} with
	\begin{align}\label{eq:moment L2}
	\E\left[\| X^{\alpha}\|^2_{L^2(0,T)}\right] < \infty,
	\end{align}
	then, it holds that
	\begin{align}
	\E\left[ \| X^{\alpha}\|^2_{L^2(0,T)} \right] &\leq   c m_T(g_0,K,\alpha) \left( 1+  \|R\|_{L^1(0,T)}\right) , \label{eq:apriori1}
	\end{align}
	where $c$ is a constant only depending on $(T,B,C,D,F)$,
	\begin{align*}
	m_T(g_0,K,\alpha)&=    \|g_0\|_{L^2(0,T)}^2 +  \|K\|_{L^2(0,T)}^2 \left( \|\beta\|^2_{L^2(0,T)} + \|\gamma\|^2_{L^2(0,T)}+ \E\big[\| \alpha\|^2_{L^2(0,T)}\big] \right)
	\end{align*}
	and $R$ is the resolvent of $c |K|^2$.
\end{lemma}


\begin{proof}
	Throughout the proof, we make use of the notations $(f*g)(t)=\int_0^t f(t-s)g(s)ds$ and $(f*dZ)_t=\int_0^t f(t-s)dZ_s$, and $c$ will denote a constant depending exclusively on $(T,B,C,D,F)$ that may vary from line to line.   We first observe that by Jensen's inequality
	\begin{align*}
	\|X^{\alpha}_t\|^2_{L^2(0,T)} &\leq  5\| g_0\|^2_{L^2(0,T)} + 5 \| K*\left(\beta + C\alpha \right) \|^2_{L^2(0,T)}    +   5\|K* BX^{\alpha}\|^2_{L^2(0,T)}\\
	&\quad + 5\| K* DX^{\alpha} dW\|^2_{L^2(0,T)} +    5 \|K*\left(\gamma  + F\alpha \right)dW\|^2_{L^2(0,T)}\\
	&=  5(\textbf{I} + \textbf{II} + \textbf{III}+\textbf{IV}+\textbf{V}).
	\end{align*}
	An application of Young and Cauchy--Schwarz inequalities yields
	\begin{align*}
	\textbf{II} 
	&\leq c \|K\|_{L^2(0,T)}^2 \left( \|\beta\|^2_{L^2(0,T)} + \| \alpha\|^2_{L^2(0,T)} \right).
	\end{align*}
	Successive applications of Cauchy--Schwarz, Tonelli's theorem and  changes of variables lead to 
	\begin{align*}
	\textbf{III}  & \leq c \int_0^T  \int_0^t |K(t-s)|^2 |X^{\alpha}_s|^2 ds dt \\
	& = c \int_0^T  \int_0^t |K(s)|^2 |X^{\alpha}_{t-s}|^2 ds dt \\
	& = c \int_0^T |K(s)|^2  \int_s^T  |X^{\alpha}_{t-s}|^2 dt ds\\
	& = c \int_0^T |K(T-s)|^2  \int_0^s  |X^{\alpha}_{u}|^2 du ds\\
	& = c \int_0^T |K(T-s)|^2   \|X^{\alpha}\|^2_{L^2(0,s)}  ds.
	\end{align*}
	Taking the expectation, we get
	\begin{align*}
	\E[\textbf{II}+\textbf{III}]
	&\leq c \|K\|_{L^2(0,T)}^2 \left( \|\beta\|^2_{L^2(0,T)} + \E\left[\| \alpha\|^2_{L^2(0,T)}\right] \right) \\
	&\quad +  c \int_0^T |K(T-s)|^2   \E\left[\|X^{\alpha}\|^2_{L^2(0,s)}\right]  ds.
	\end{align*}
	Similarly,  It\^o's isometry combined with Tonelli's theorem and multiple changes of variables give
	\begin{align*}
	\E[\textbf{IV}]&=|D|^2 \E\left[ \int_0^T  \int_0^t |K(t-s)|^2 |X^{\alpha}_s|^2 ds  dt\right] \\
	& = |D|^2  \int_0^T |K(T-s)|^2 \E\left[\|X^{\alpha}\|^2_{L^2(0,s)}\right] ds. 
	\end{align*}
	Another application of It\^o's isometry and Young's inequality  shows that 
	\begin{align*}
	\E[\textbf{V}]&  \leq c \| |K|^2*\left(|\gamma|^2  + |F|^2|\alpha|^2 \right)\|^2_{L^1(0,T)}   \\
	& \leq c \| K\|_{L^2(0,T)}^2 \left( \| \gamma\|_{L^2(0,T)}^2 +\E\left[ \| \alpha\|_{L^2(0,T)}^2\right]  \right).
	\end{align*}
	Combining the above yields
	\begin{align*}
	\E\left[\|X^{\alpha}\|^2_{L^2(0,T)}\right] 
	&\leq  c \|g_0\|_{L^2(0,T)}^2 + c \|K\|_{L^2(0,T)}^2 \left( \|\beta\|^2_{L^2(0,T)} + \|\gamma\|^2_{L^2(0,T)}+ \E\left[\| \alpha\|^2_{L^2(0,T)}\right] \right) \\
	&\quad +  c \int_0^T |K(T-s)|^2   \E\left[\|X^{\alpha}\|^2_{L^2(0,s)}\right]  ds,\\
	&\leq c m_T(g_0,K,\alpha)\left( 1+  \|R\|_{L^1(0,T)}\right)
	\end{align*}
	where the last line follows from  the generalized Gronwall inequality for convolution equations with $R$  the resolvent of $c |K|^2$, see \cite[Theorem 9.8.2]{GLS:90}. 
\end{proof}

\begin{lemma}
	\label{lemma:XXn_approximation}
	Fix $n \in \N$. Let $K,K^n\in L^2([0,T],\R^{d\times d'})$, $\alpha \in \Acal$,   $g_0,g_0^n \in L^2([0,T],\R^d)$ 	and $\beta, \gamma \in L^2([0,T],\R^{d'})$. 
	Assume that there exist  two progressively 
	measurable processes  $X$ and $X^n$ satis\-fying \eqref{eq:sve}  and \eqref{eq:moment L2} for the respective inputs  $(g_0,K,\alpha)$ and $(g_0^n,K^n, \alpha)$. Then, 
	\begin{align}
	\label{eq:estimatediffL2}
	\E\left[  \|X^n -X\|^2_{L^2(0,T)}\right] \leq  cm_n \left(1 + \|R^n\|_{L^1(0,T)}\right),
	\end{align}
	where 
	\begin{align}
	m_n &=  \|g_0^n- g_0\|^2_{L^2(0,T)} +  \|K^n-K\|^2_{L^2(0,T)}\left(  \E\left[\|X\|^2_{L^2(0,T)}\right] + \E\left[\|\alpha\|^2_{L^2(0,T)}\right] \right)  
	\label{eq:estimatediffmnL2}
	\end{align}
	and $R^n$ is the resolvent of $c |K^n|^2$. If in addition 
	\begin{align}
	\label{eq:estimatediffconvL2} \|g_0^n-g_0 \|_{L^2(0,T)}\to 0,   \quad \|K^n-K\|_{L^2(0,T)}\to 0,
	\end{align}
	as $n\to \infty$, then, 
	\begin{align}\label{eq:convXnXL2}
	\E\left[ \|X^n-X \|_{L^2(0,T)}^2\right] \to 0,  \quad \mbox{as $n\to \infty$}.
	\end{align}
\end{lemma}

\begin{proof}
	Fix $t \leq T$. We start by writing 
	\begin{align*}
	X_t- X^n_t &=\left(g_0(t)-g_0^n(t)\right)+ \int_0^t \left(K(t-s)-K^n(t-s)\right)\left( BX_s + C\alpha_s\right) ds \\
	&\quad - \int_0^tK^n(t-s) B (X^n_s-X_s)   ds \\
	&\quad + \int_0^t \left(K(t-s)-K^n(t-s)\right)\left( DX_s + F\alpha_s\right) dW_s \\
	&\quad - \int_0^tK^n(t-s) D (X^n_s-X_s)  dW_s \\
	&= \textbf{I}_t + \textbf{II}_t + \textbf{III}_t + \textbf{IV}_t + \textbf{V}_t.
	\end{align*}
	In the sequel, $c$ denotes a constant independent of $n$ that may vary from line to line.  Repeating the same argument as in the proof of Lemma \ref{L:aprioriL2}, we get
	\begin{align*}
	\E\left[ \|\textbf{II}\|^2_{L^2(0,T)} + \|\textbf{IV}\|^2_{L^2(0,T)}\right]\leq c\|K^n-K\|_{L^2(0,T)}^2 \left(  \E\left[\|X\|^2_{L^2(0,T)}\right] + \E\left[\|\alpha\|^2_{L^2(0,T)}\right] \right),
	\end{align*}
	which is finite due to Lemma~\ref{L:aprioriL2}.
	Similarly,
	\begin{align*}
	\E \left[\|\textbf{III}\|^2_{L^2(0,T)}+\|\textbf{V}\|^2_{L^2(0,T)}\right]&\leq	\; c \int_0^T |K^n(T-s)|^2 \E\left[\|X^n -X\|^2_{L^2(0,s)}\right]ds.
	\end{align*}
	Combining the above and invoking \cite[Theorem 9.8.2]{GLS:90} for the gene\-ralized Gronwall inequality for convolution equations yields the estimate \eqref{eq:estimatediffL2}.  We now prove that its right hand side goes to $0$, as $n$ goes to infinity. We first note that $R^n \to R$ in $L^1$, by virtue of the continuous dependence of the resolvent on the kernel combined with the $L^2$--convergence of $(K^n)_{n \geq 1}$ in \eqref{eq:estimatediffconvL2},  see \cite[Lemma 9.3.11]{GLS:90}. Consequently, the sequences $(\|R^n\|_{L^1(0,T)})_{n \geq 1}$ and  $(\|K^n\|_{L^2(0,T)})_{n \geq 1}$ are uniformly bounded in $n$, that is 
	\begin{align}\label{eq:supRn}
	\sup_{n \geq 1}  \|R^n\|_{L^1(0,T)} +  \sup_{n \geq 1}  \|K^n\|_{L^2(0,T)} < \infty.
	\end{align}
	Thus,  it follows from \eqref{eq:estimatediffL2} that it is enough to prove that $m_n \to 0$ to get the claimed convergence \eqref{eq:convXnXL2}. This is straightforward  from \eqref{eq:estimatediffconvL2} and the proof is complete.
\end{proof}

\subsection{Approximation of the value function}\label{S:approximationproof}
The proof of Theorem~\ref{T:mainstability} now follows from the two following lemmas.  In the sequel, we work under the assumptions of   Theorem~\ref{T:mainstability} and we recall the expressions of  $J^n,X^{n,\alpha}$ in \eqref{eq:approxprob}. To ease notations, we drop the $\alpha$ superscripts.

\begin{lemma}
	\label{L:diffJ}
	Let $\c \in \Acal$. Under \eqref{assumption:approximation} we have 
	\bes{
		|J(\c) - J^n(\c)|^2 \leq c \left(2+ \left(\E\left[\| \alpha\|^2_{L^2(0,T)}\right] \right)^2 \right)  \left( \|g_0^n- g_0\|^2_{L^2(0,T)}  + \|K^n-K\|^2_{L^2(0,T)}   \right),
	}
	where $c$ is a constant independent of $n$.
\end{lemma}
\begin{proof}
	Fix $\c \in \Acal$, we start by writing
	\begin{align*}
	J(\c) - J^n(\c) &=  \E\left[\int_0^T \left(X^\T_s Q X_s-(X_s^n)^\T Q X^n_s \right)ds\right] +  \E\left[\int_0^T (X_s - X^n_s)^\T L ds\right]\\
	&= \E\left[\int_0^T (X_s-X^n_s)^\T Q (X_s+X^n_s) ds\right] +  \E\left[\int_0^T (X_s - X^n_s)^\T L ds\right]\\
	&=  \textbf{I} + \textbf{II}
	\end{align*}
	so that 
	\bes{
		|J(\c) - J^n(\c)|^2 & \leq 2\left(\textbf{I}^2 + \textbf{II}^2\right).
	}
	We let $c$ denote a constant independent of $n$ that may vary from line to line. Successive  applications of Cauchy-–Schwarz inequality and Lemma   \ref{lemma:XXn_approximation} yield
	\bes{
		\textbf{II}^2 &\leq c \E\left[\|X-X^n\|^2_{L^2(0,T)} \right] \\
		& \leq c  \left(1+\|R^n\|_{L^1(0,T)}\right) \left(1 + \E\left[\|\alpha\|^2_{L^2(0,T)}\right]   \right) \left( \|g_0^n- g_0\|^2_{L^2(0,T)} + \|K^n-K\|^2_{L^2(0,T)} \right)
	}
	where $R^n$ is the resolvent of $c|K^n|^2$. By virtue of the $L^2$ convergence of the kernels $(K^n)_{n\geq 1}$ in \eqref{assumption:approximation},  $\|R^n\|_{L^1(0,T)}$ is uniformly bounded in $n$, see \eqref{eq:supRn}. 
	Whence,
	\bes{
		\textbf{II}^2
		& \leq c   \left(1 + \E\left[\|\alpha\|^2_{L^2(0,T)}\right]   \right) \left( \|g_0^n- g_0\|^2_{L^2(0,T)} + \|K^n-K\|^2_{L^2(0,T)} \right).
	} 
	Similarly, we get from Lemmas \ref{L:aprioriL2}   and \ref{lemma:XXn_approximation}
	\bes{
		\textbf{I}^2 
		& \leq c \left( \E \left[ \|X\|^2_{L^2(0,T)}\right] + \E \left[ \|X^n\|^2_{L^2(0,T)}\right] \right) \E\left[\|X-X^n\|^2_{L^2(0,T)} \right] \\
		& \leq c \left(1 + \left(\E\left[\|\alpha\|^2_{L^2(0,T)}\right]\right)^2    \right) \left( \|g_0^n- g_0\|^2_{L^2(0,T)} + \|K^n-K\|^2_{L^2(0,T)} \right),
	}
	where the last inequality  follows from the fact that  $\sup_{n\geq 1 }\E \left[ \|X^n\|^2_{L^2(0,T)}\right]<\infty$, since $\E\left[ \|X^n-X \|_{L^2(0,T)}^2\right] \to 0$ from Lemma \ref{lemma:XXn_approximation}.
	Combining the above yields the desired estimate. 
\end{proof}

\begin{lemma}
	\label{lemma:bounded_optimal_control}
	Assume  \eqref{assumption:QN}, \eqref{assumption:approximation} and that $Q$ is invertible. Let $\alpha^*$ and $\alpha^{n*}$ be the optimal controls produced by Theorem~\ref{T:mainoptimalsolution}  respectively for the problem \eqref{eq:original_problem} and its approximation \eqref{eq:approx_n}.  There exists a constant $\kappa > 0$ such that 
	\bes{\label{eq:alphanball}
		\E\left[\|\alpha^*\|^2_{L^2(0,T)}\right] + \sup_{n\geq 1}\E\left[\|\alpha^{n*}\|^2_{L^2(0,T)}\right] \leq \kappa.
	}
\end{lemma}
\begin{proof}
	{ Under \eqref{assumption:QN},  there exists $c>0$ such that 
		$$ |a|^2 \leq c a^\top N a, \quad a \in \R^m. $$ 
		Denoting by $X^n=X^{n,\alpha^{n*}}$, it follows that
		\begin{align*}
		\E \left[ \|\c^{n*}\|^2_{L^2(0,T)}\right] &\leq   (1\vee c)\E \left[ \int_0^T \left( \left(\alpha_s^{n*}\right)^\top N \alpha_s^{n*} + \left(X_s^{n} + Q^{-1}L \right)^\top Q \left(X_s^{n} + Q^{-1}L \right) \right) ds\right]  \\
		& =   (1\vee c) \left(J^n(\alpha^{n*}) +  L^\top Q^{-1} L\right)\\
		&\leq (1\vee c) \left( J^n(\bold 0) + L^\top Q^{-1} L\right),
		\end{align*}
		for all $n \in \mathbb N$,
		where the last inequality follows from the optimality of $\alpha^{n*}$ and $\bold 0$ corresponds to the admissible control $\alpha_s= 0$, for all $s\leq T$.
	}
	Applying Lemma \ref{L:diffJ}, with  $\alpha=\alpha^n=\bold 0$, we obtain the convergence of the un-controlled functional cost:  $\lim_{n \to \infty} J^n(\bold 0)$ $=$ $J(\bold 0)$, 
	which ensures that $J^n(\bold 0)$ is uniformly bounded in $n$. We then deduce the existence of a constant $\kappa$ such that \eqref{eq:alphanball} holds.
\end{proof}

\vspace{1mm}

The proof of Theorem~\ref{T:mainstability} is now straightforward.

\begin{proof}[Proof of Theorem~\ref{T:mainstability}]
	Fix an arbitrary $\varepsilon >0$,  and let $c_\varepsilon>0$ to be determined later.
	First note that Lemma \ref{lemma:bounded_optimal_control} ensures  the existence a constant $\kappa >0$ such that for all $n\in\N$:
	\bes{	\label{eq:inf_in_ball}
		V^n_0 &= \inf_{\alpha \in \Acal} J^n(\alpha) = \inf_{\alpha \in \Acal_\kappa} J^n( \alpha), \\
		V_0 &= \inf_{\alpha \in \Acal} J( \alpha) = \inf_{\alpha \in \Acal_\kappa} J( \alpha), 
	}
	where $\Acal_\kappa$ $=$ $\{  \alpha \in \Acal : \E[\|\c\|^2_{L^2(0,T)}] \leq \kappa\}$.  Under condition \eqref{assumption:approximation}, there exists 
	${n_{\varepsilon} \in \N}$ such that for every $n \geq n_{\varepsilon}$ we have $ \|g_0^n- g_0\|^2_{L^2(0,T)}  + \|K^n-K\|^2_{L^2(0,T)} \leq c_\varepsilon$. By  Lemma \ref{L:diffJ}, it follows that 
	for any  $\c \in \Acal_\kappa$, and $n$ $\geq$ $n_{\varepsilon}$, 
	\begin{align*}
	|J(\c) - J^n(\c)|^2 &  \leq  c(2 + \kappa^2) c_\varepsilon \; = \;  \varepsilon,  
	\end{align*}
	by choosing $c_\varepsilon= \frac{\epsilon}{c(2+ \kappa^2)}$.  Combined with \eqref{eq:inf_in_ball}, this gives \eqref{eq:convvaluefunction} and also \eqref{convVn}.          
\end{proof}

\appendix

\section{An elementary lemma}


\begin{lemma}
	\label{L:L2kernel}
	Let $K$ be given as in  \eqref{eq:cmmu}, and $\bar K$  defined by 
	\begin{align}
	\bar K(t) &= \int_{\R_+} e^{-\theta t} |\mu|(d\theta), \quad t>0. 
	\end{align}
	Assume  that \eqref{eq:totalvar} holds, then 
	\begin{align*}
	\int_0^T |K(s)|^2 ds  \leq   \int_0^T \bar K(s)^2 ds < \infty.
	\end{align*}
	Furthermore, $|\mu|$ is $\sigma$-finite.
\end{lemma}

\begin{proof}
	Since, $|K(t)| \leq \bar K(t)$, for all $ t >0$, it is clear that $\int_0^T |K(s)|^2 ds \leq \int_0^T \bar K(s)^2ds$.  Furthermore, 
	\begin{align*}
	\|\bar K\|_{L^2(0,T)}&=   \left\|\int_{\R_+} e^{-\theta(\cdot)} |\mu|(d\theta) \right\|_{L^2(0,T)}\\ &\leq  \int_{\R_+} \left\| e^{-\theta(\cdot)}  \right\|_{L^2(0,T)} |\mu|(d\theta)
	\; =  \; \int_{\R_+} \sqrt{\frac {1-e^{-2\theta T}}{2\theta}} |\mu|(d\theta)
	\end{align*}
	which is finite due to  inequality  \eqref{eq:inequalityexp} and condition \eqref{eq:totalvar}. To prove that $|\mu|$ is $\sigma$-finite, we observe  that $\R_+=\cup_{n \in \mathbb N} [0,n]$ such that for each $n\geq 1$,
	\begin{align*}
	|\mu|([0,n]) \; = \; \int_0^1    |\mu|(d\theta) + \int_1^n  |\mu|(d\theta)  &\leq  \int_0^1    |\mu|(d\theta) + \sqrt{n} \int_1^n  \theta^{-1/2} |\mu|(d\theta)\\
	&\leq   \sqrt{n}  \int_{\R_+}  (1\wedge \theta^{-1/2})  |\mu|(d\theta)<\infty.
	\end{align*}
\end{proof}


\small



\bibliographystyle{abbrv}

\bibliography{bibl}

\begin{thebibliography}{10}

\bibitem{abi2019lifting}
E.~Abi~Jaber.
\newblock Lifting the {H}eston model.
\newblock {\em Quantitative Finance}, pages 1--19, 2019.

\bibitem{AJEE18b}
E.~Abi~Jaber and O.~El~Euch.
\newblock Markovian structure of the {V}olterra {H}eston model.
\newblock {\em Statistics \& Probability Letters}, 149:63--72, 2019.

\bibitem{AJEE18a}
E.~Abi~Jaber and O.~El~Euch.
\newblock Multifactor approximation of rough volatility models.
\newblock {\em SIAM Journal on Financial Mathematics}, 10(2):309--349, 2019.

\bibitem{AJLP17}
E.~Abi~Jaber, M.~Larsson, and S.~Pulido.
\newblock Affine {V}olterra processes.
\newblock {\em The Annals of Applied Probability}, 29(5):3155--3200, 2019.

\bibitem{abietal19b}
E.~Abi~Jaber, E.~Miller, and H.~Pham.
\newblock Integral operator {R}iccati equations arising in stochastic
  {V}olterra control problems.
\newblock {\em arXiv:1911.01903}, 2019.

\bibitem{agram2015malliavin}
N.~Agram and B.~{\O}ksendal.
\newblock Malliavin calculus and optimal control of stochastic {V}olterra
  equations.
\newblock {\em Journal of Optimization Theory and Applications},
  167(3):1070--1094, 2015.

\bibitem{alfonsi2013capacitary}
A.~Alfonsi and A.~Schied.
\newblock Capacitary measures for completely monotone kernels via singular
  control.
\newblock {\em SIAM Journal on Control and Optimization}, 51(2):1758--1780,
  2013.

\bibitem{bank2017hedging}
P.~Bank, H.~M. Soner, and M.~Vo{\ss}.
\newblock Hedging with temporary price impact.
\newblock {\em Mathematics and financial economics}, 11(2):215--239, 2017.

\bibitem{baretal11}
O.~Barndorff-Nielsen, F.~Benth, and A.~Veraart.
\newblock Ambit processes and stochastic partial differential equations.
\newblock In {\em Advances mathematical methods for finance, eds G. Di Nunno
  and B. Oksendal}, pages 35--74, 2011.

\bibitem{bonaccorsi2012optimal}
S.~Bonaccorsi, F.~Confortola, and E.~Mastrogiacomo.
\newblock Optimal control for stochastic {V}olterra equations with completely
  monotone kernels.
\newblock {\em SIAM Journal on Control and Optimization}, 50(2):748--789, 2012.

\bibitem{carmona1998fractional}
P.~Carmona and L.~Coutin.
\newblock Fractional {B}rownian motion and the {M}arkov property.
\newblock {\em Electronic Communications in Probability}, 3:95--107, 1998.

\bibitem{cuchiero2018generalized}
C.~Cuchiero and J.~Teichmann.
\newblock Generalized {F}eller processes and {M}arkovian lifts of stochastic
  {V}olterra processes: the affine case.
\newblock {\em arXiv preprint arXiv:1804.10450}, 2018.

\bibitem{duncan2013linear}
T.~E. Duncan and B.~Pasik-Duncan.
\newblock Linear-quadratic fractional {G}aussian control.
\newblock {\em SIAM Journal on Control and Optimization}, 51(6):4504--4519,
  2013.

\bibitem{EER:07}
O.~El~Euch and M.~Rosenbaum.
\newblock Perfect hedging in rough {H}eston models.
\newblock {\em The Annals of Applied Probability}, 28(6):3813--3856, 2018.

\bibitem{flandoli1986direct}
F.~Flandoli.
\newblock Direct solution of a {R}iccati equation arising in a stochastic
  control problem with control and observation on the boundary.
\newblock {\em Applied Mathematics and Optimization}, 14(1):107--129, 1986.

\bibitem{volatilityrough2014}
J.~Gatheral, T.~Jaisson, and M.~Rosenbaum.
\newblock Volatility is rough.
\newblock {\em Quantitative Finance}, 18(6):933--949, 2018.

\bibitem{GLS:90}
G.~Gripenberg, S.-O. Londen, and O.~Staffans.
\newblock {\em Volterra integral and functional equations}, volume~34 of {\em
  Encyclopedia of Mathematics and its Applications}.
\newblock Cambridge University Press, Cambridge, 1990.

\bibitem{han2019time}
B.~Han and H.~Y. Wong.
\newblock Time-consistent feedback strategies with {V}olterra processes.
\newblock {\em arXiv preprint arXiv:1907.11378}, 2019.

\bibitem{harms2019affine}
P.~Harms and D.~Stefanovits.
\newblock Affine representations of fractional processes with applications in
  mathematical finance.
\newblock {\em Stochastic Processes and their Applications}, 129(4):1185--1228,
  2019.

\bibitem{hu2018stochastic}
Y.~Hu and S.~Tang.
\newblock Stochastic {LQ} and associated {R}iccati equation of {PDE}s driven by
  state-and control-dependent {W}hite noise.
\newblock {\em arXiv preprint arXiv:1809.05308}, 2018.

\bibitem{jacomg19}
A.~Jacquier and M.~Oumgari.
\newblock Deep {PPDE}s for rough local stochastic volatility.
\newblock {\em arXiv:1906.02551}, 2019.

\bibitem{klepetal03}
M.~Kleptsyna, A.~L. Breton, and M.~Viot.
\newblock About the linear quadratic regulator problem under a fractional
  {B}rownian perturbation,.
\newblock {\em ESAIM Probab. Stat}, 9:161--170, 2003.

\bibitem{MS:15}
L.~Mytnik and T.~S. Salisbury.
\newblock Uniqueness for {V}olterra-type stochastic integral equations.
\newblock {\em arXiv preprint arXiv:1502.05513}, 2015.

\bibitem{PP:90}
E.~Pardoux and P.~Protter.
\newblock Stochastic {V}olterra equations with anticipating coefficients.
\newblock {\em Ann. Probab.}, 18(4):1635--1655, 1990.

\bibitem{rudin2006real}
W.~Rudin.
\newblock {\em Real and complex analysis}.
\newblock Tata McGraw-hill education, 2006.

\bibitem{sch06}
J.~Schmiegel.
\newblock Self-scaling tumor growth.
\newblock {\em Physica A: Statistical Mechanics and its Applications,},
  367(C):509--524, 2006.

\bibitem{V:12}
M.~Veraar.
\newblock The stochastic {F}ubini theorem revisited.
\newblock {\em Stochastics}, 84(4):543--551, 2012.

\bibitem{viezha18}
F.~Viens and J.~Zhang.
\newblock A martingale approach for fractional {B}rownian motions and related
  path dependent {PDE}s.
\newblock {\em Annals of Applied Probability, to appear}, 2018.

\bibitem{wan18}
T.~Wang.
\newblock Linear quadratic control problems of stochastic {V}olterra integral
  equations.
\newblock {\em ESAIM: Control, Optimisation and Calculus of Variations},
  24:1849--1879, 2018.

\bibitem{yong2006backward}
J.~Yong.
\newblock Backward stochastic {V}olterra integral equations and some related
  problems.
\newblock {\em Stochastic Processes and their Applications}, 116(5):779--795,
  2006.

\bibitem{yong1999stochastic}
J.~Yong and X.~Y. Zhou.
\newblock {\em Stochastic controls: Hamiltonian systems and HJB equations},
  volume~43.
\newblock Springer Verlag, 1999.

\end{thebibliography}

\end{document}